\documentclass{amsart}

\usepackage{amsmath,amsxtra,amsthm,amssymb,mathtools,stmaryrd}
\usepackage[b]{esvect}
\usepackage{tikz-cd}
\usepackage{tcolorbox}
\usepackage{hyperref}

\hypersetup{colorlinks=true, linkcolor=blue}

\newtheorem{thm}{Theorem}[section]
\newtheorem{prop}[thm]{Proposition}
\newtheorem{lem}[thm]{Lemma}
\newtheorem{cor}[thm]{Corollary}
\newtheorem{rem}[thm]{Remark}
\newtheorem{exa}[thm]{Example}
\numberwithin{equation}{section}

\DeclareMathOperator{\Ker}{Ker}

\DeclareMathOperator{\Hom}{Hom}
\DeclareMathOperator{\End}{End}
\DeclareMathOperator{\Ext}{Ext}
\DeclareMathOperator{\Tor}{Tor}
\DeclareMathOperator{\Aut}{Aut}

\DeclareMathOperator{\modcat}{mod}
\DeclareMathOperator{\Mod}{Mod}
\DeclareMathOperator{\fd}{fd}
\DeclareMathOperator{\gen}{gen}
\DeclareMathOperator{\nilp}{nilp}
\DeclareMathOperator{\HOM}{HOM}
\DeclareMathOperator{\EXT}{EXT}
\DeclareMathOperator{\LL}{LL}
\DeclareMathOperator{\Der}{Der}
\DeclareMathOperator{\add}{add}

\DeclareMathOperator{\id}{id}
\DeclareMathOperator{\gldim}{gl.dim}
\DeclareMathOperator{\soc}{soc}
\DeclareMathOperator{\rad}{rad}

\newcommand{\bZ}{\mathbb Z}
\newcommand{\gr}{\mathrm{gr}}
\newcommand{\hyp}{\textrm{-}}
\newcommand{\op}{\mathrm{op}}

\newcommand{\iso}{\xrightarrow{\raisebox{-.5ex}[0ex][0ex]{$\scriptstyle{\sim}$}}}
\newcommand{\leftiso}{\xleftarrow{\raisebox{-.5ex}[0ex][0ex]{$\scriptstyle{\sim}$}}}

\newcommand*{\intref}[2]{\hyperref[#2]{#1~\ref*{#2}}}
\newcommand{\mat}[1]{\mathsf{#1}}
\newcommand{\rmod}[1]{\vv{#1}}
\newcommand{\lmod}[1]{\reflectbox{$\rmod{\reflectbox{$#1$}}$}}

\title[Global dimension and Koszul dual]{On the global dimension and Koszul property of preprojective algebras}
\author{Andrew Hubery}
\address{Bielefeld University\\33501 Bielefeld\\Germany}
\email{hubery@math.uni-bielefeld.de}

\begin{document}

\begin{abstract}
We study preprojective algebras associated to either finite dimensional hereditary algebras, or locally finite hereditary tensor algebras, and in particular show that they have global dimension two in non-Dynkin type. Moreover, starting from a locally finite hereditary tensor algebra of non-Dynkin type, we show that the corresponding preprojective algebra is Koszul, and compute both its Hilbert polynomial and its Koszul dual. We finish by looking at preprojective algebras of Dynkin type, and show how certain properties of the Weyl group arise in the structure of the preprojective algebra.
\end{abstract}

\maketitle

\section{Introduction}

The preprojective algebras of finite quivers form an interesting class of algebras having connections to many areas of mathematics. These include, the MacKay correspondence and Kleinian singularities (for affine quivers) \cite{RVdB,CBH}, Lusztig's semicanonical basis and cluster theory \cite{Lusztig,GLS} (see also \cite{Leclerc} for a nice overview), Nakajima's quiver varieties \cite{Nakajima} (see also \cite{Nakajima2}). More recently there has been emphasis on the fact that they are 2-Calabi-Yau algebras (in the non-Dynkin case), their Hochschild (co)homology has been computed in the Dynkin case \cite{EE2}, and their Ksozul properties have been studied \cite{BBK,EE,MV}. In terms of non-commutative geometry, the preprojective algebra can be regarded as the Hamiltonian reduction of the non-commutative cotangent bundle associated to the (smooth) quiver path algebra \cite{CBEG}. Preprojective algebras were also used to solve the additive Deligne--Simpson Problem \cite{CB2}, and certain deformations of them played a central role in the solution to the original, multiplicative, Deligne--Simpson Problem \cite{CBHu,Hubery}.

All of these properties have been shown when working with an acyclic quiver over a field, usually of characteristic zero, and possibly algebraically closed. Their study for general quivers is more limited, and if we work instead with an arbitrary locally finite hereditary tensor algebra then very little has been proved. On the other hand, the original definition of the preprojective algebra in \cite{BGL} was for a general finite dimensional hereditary algebra, and such algebras are not necessarily tensor algebras. In this case very little is known.

In this first of a series of papers we investigate some of the basic properties of preprojective algebras attached to hereditary algebras which are either finite dimensional, or else locally finite tensor algebras.

We will first show that the preprojective algebra has two natural constructions, as a tensor algebra $T_\Lambda(\Pi_1)$ over the hereditary algebra, and as the quotient of an hereditary tensor algebra $T_\Lambda(\mho)$ by a single element $\bar c$. If $\Lambda$ is finite dimensional, then $\Pi_1\cong\tau^-\Lambda$ and we recover the original definition of \cite{BGL}; if instead $\Lambda=kQ$ is the path algebra of a quiver, then $T_\Lambda(\mho)\cong k\bar Q$ is isomorphic to the path algebra of the double quiver, $\bar c$ corresponds to the sum of commutators $\sum[a,a^\ast]$ over all arrows $a\in Q$ and their duals, and we recover the usual definition.

We then proceed to construct in \intref{Theorem}{thm:standard-proj-pi} a standard (functorial) projective presentation of length three for each module over the preprojective algebra, and show in \intref{Theorem}{thm:gl-dim-2} that when we are non-Dynkin type, then this is in fact a projective resolution. It follows that the preprojective algebra always has global dimension two. For quivers this result seems to be known, but an explicit reference has been hard to track down. In \cite[Proposition 4.2]{BBK} they show that, for an acyclic quiver of non-Dynkin type, each graded simple module over the preprojective algebra has projective dimension at most two; in \cite{EE} they generalise this to cover all non-Dynkin quivers, and further show that the preprojective algebra is Koszul. In \cite{BGL} they show that the preprojective algebra of a finite dimensional tame hereditary algebra has global dimension two. For finite dimensional wild hereditary algebras and general hereditary tensor algebras, though, nothing is known.

The proof in the setting of a finite dimensional hereditary algebra is actually quite straightforward. The left hand map of the standard projective presentation of $\Pi$, which we need to show is a monomorphism, can be composed with another natural map such that the kernel, viewed as a $\Lambda$-module, is a direct sum of homomorphisms from postinjective modules to preprojective ones, so vanishes. In the setting of an hereditary tensor algebra, the aim is therefore to generalise this result, which we do by analysing the category of graded modules.

Having shown that the preprojective algebra has global dimension two, we can then use the standard projective presentation, in the setting of an hereditary tensor algebra, to show that there is a natural isomorphism of bifunctors
\[ \Ext^2_\Pi(X,Y) \cong D\Hom_\Pi(Y,X) \]
for all finite dimensional modules $X$ and $Y$. Again, this generalises the known result for path algebras. In the setting of a finite dimensional hereditary algebra we have the same result under an invariance assumption on the Auslander--Reiten formula.

We then turn our attention to the Koszul property, so necessarily restricted to the setting of a locally finite hereditary tensor algebra. It is clear that, in the non-Dynkin case, our standard projective resolution of the semisimple quotient $A$ is a linear resolution, and hence that the preprojective algebra is always Koszul. We also compute its Hilbert series. For quivers this recovers the results in \cite{EE}, but the general case seems to be new. We then proceed to compute the Koszul dual, and offer alternative descriptions as either a trivial extension algebra of the radical-square quotient $\Lambda/\Lambda_+^2$, or as an explicit quadratic algebra. For a radical-squared zero path algebra $kQ$ over an algebraically closed field $k$, the Koszul dual of the preprojective algebra was computed in \cite{MV} as a trivial extension algebra of $kQ$. Comparing the two descriptions we see that Mart\'inez-Villa description as $(kQ)\ltimes D(kQ)$ is slightly misleading, as in general we should really add a sign-twist the bimodule $D(kQ)$ on one side. Of course, this only makes a difference when we start with a path algebra $kQ$ which is not radical-squared zero.

For completeness we finish with a discussion of preprojective algebras of Dynkin type. In this case the preprojective algebra is finite dimensional, self-injective, and the third syzygy functor is naturally isomorphic to the inverse Nakayama functor. Again these results are well-known, but references are hard to find. For quivers there is a proof in \cite{BBK}, whereas Ringel and Schofield have an unpublished proof in general using hammocks. Recently Grant \cite{Grant} gave an alternative proof of self-injectivity, similar to ours. We also show that the preprojective algebra has Loewy length $h-1$, where $h$ is the associated Coxeter number, and that, in the language of \cite{BBK}, the preprojective algebra is $(h-2,2)$-Koszul. This was already shown in \cite{BBK} in the simply-laced cases, and was recently proved by S\"oderberg in general \cite{Sb}, though his techniques are quite different.

Moreover, we can relate the matrix of dimensions of each graded piece $\Pi_n$ to the Vieta--Fibonacci polynomials, so in particular we see that $V_{h-1}(\bar{\mat B})=0$. Such connections between symmetrisable generalised Cartan matrices, Vieta--Fibonacci polynomials, and representation theory have been investigated previously by several authors, including for example \cite{LdlP} and more recently \cite{ES}. In particular, the matrices $V_r(\bar{\mat B})$ were computed in \cite{ES} in the simply-laced finite and affine cases. Thus our results, which follow from the Koszul property of the preprojective algebra, provide a representation-theoretic proof of several of their computations, and at the same time extend them to the non-simply laced cases.

Similarly, we obtain a representation-theoretic proof that the action of a Coxeter element on the set of roots has precisely $n$ orbits, each of size $h$, where $n$ is the rank. In particular, the number of roots is always $nh$. Interestingly, the proof of this fact in Bourbaki's classic text \cite[VI \S1.11 Proposition 33]{Bourbaki} relies on computing the eigenvalues of the Coxeter transformation \cite[V \S6.2]{Bourbaki} which in turn is done by exhibiting the the Dynkin diagram as a bipartite graph. This is completely analogous to our proof, which relies on choosing an alternating orientation of the Dynkin diagram, so in some sense we have use the preprojective algebra to categorify the original, more combinatorial, proof.

\section{Hereditary algebras}

We will concentrate on two types of hereditary $k$-algebras --- those which are finite dimensional, and those which are locally finite tensor algebras. There is of course some overlap between these classes, and if the base field $k$ is perfect, then every finite dimensional hereditary algebra is a tensor algebra. In general, though, there are hereditary algebras which are finite dimensional but not tensor algebras, and conversely every quiver containing an oriented cycle yields a locally finite hereditary tensor algebra which is not finite dimensional.

We will often be working with $\Lambda$-bimodules for some $k$-algebra $\Lambda$, in which case we will always assume that $k$ acts centrally. Thus for us a $\Lambda$-bimodule will be the same as a right module over the enveloping algebra $\Lambda^e=\Lambda^\op\otimes_k\Lambda$. We will write $\Hom_{\Lambda\hyp\Lambda}(X,Y)$ for the space of bimodule homomorphisms, as well as $\Hom_{\Lambda\hyp}(X,Y)$ and $\Hom_{\,\hyp\Lambda}(X,Y)$ when we consider just one sided module homomorphisms. If the context is clear, then we will simply write $\Hom_\Lambda(X,Y)$.

\subsection{Tensor algebras}

Let $\Gamma$ be a finite dimensional $k$-algebra, and $B$ a finite dimensional $\Gamma$-bimodule. We may then form the tensor algebra
\[ T = T_\Gamma(B) = \Gamma \oplus B \oplus (B\otimes_\Gamma B) \oplus (B\otimes_\Gamma B\otimes_\Gamma B) \oplus \cdots \]
and observe that this is a locally finite graded algebra. The multiplication on $T$ is induced by concatenation of tensors.

The universal property of the tensor algebra tells us that the category of right $T$-modules is equivalent to the category whose objects are pairs consisting of a right $\Gamma$-module $X$ and a $\Gamma$-linear map $X\otimes_\Gamma B\to X$.

We now have the following well-known result.

\begin{prop}\label{prop:tensor-alg-seq}
Let $T=T_\Gamma(B)$. Then there is the fundamental exact sequence of $T$-bimodules
\[ \begin{tikzcd}
0 \arrow[r] & T\otimes_\Gamma B\otimes_\Gamma T \arrow[r] & T\otimes_\Gamma T \arrow[r] & T \arrow[r] & 0
\end{tikzcd} \]
where the right hand map is multiplication, and the left hand map sends $1\otimes b\otimes 1$ to $1\otimes b-b\otimes 1$.
\end{prop}

\begin{proof}
Decomposing $T=\Gamma\oplus(B\otimes_\Gamma T)$ as $\Gamma$-bimodules, there is a locally nilpotent endomorphism $\theta$ of $T\otimes_\Gamma T$ as a right $\Gamma$-module sending $t\otimes(\gamma'+b\otimes\gamma)$ to $tb\otimes\gamma$. The desired short exact sequence can then be realised by taking this split exact sequence of right $\Gamma$-modules and then twisting in the middle by the automorphism $1-\theta$.
\end{proof}

\begin{cor}\label{co:standard-proj-tensor}
Assume $A$ is semisimple. Then $\Lambda=T_A(M)$ is hereditary, and every right module $X$ admits a standard projective resolution
\[ \begin{tikzcd}
0 \arrow[r] & X\otimes_AM\otimes_A\Lambda \arrow[r] & X\otimes_A\Lambda \arrow[r] & X \arrow[r] & 0.
\end{tikzcd} \]
\end{cor}

\begin{proof}
The sequence arises by tensoring the fundamental exact sequence for $\Lambda$ with $X$. As $A$ is semisimple, the first and second terms are projective $\Lambda$-modules.
\end{proof}

\begin{exa}
If $A=k^n$, then every $A$-bimodule is a direct sum of one dimensional bimodules. In this case $\Lambda$ is isomorphic to the path algebra of a quiver $kQ$, where the arrows yield a basis for the simple direct summands of $M$.

Furthermore, an $A$-module is just a tuple of vector spaces, and consequently the category of pairs introduced above is equivalent to the category of quiver representations.
\end{exa}

\subsection{Finite dimensional algebras}

Let $\Lambda$ be a finite dimensional hereditary algebra, and write $J$ for its Jacobson radical. We recall the following structural result from \cite{Chase,Zaks2}. 

\begin{prop}
The natural map $\Lambda\to\Lambda/J$ splits, so we can write $\Lambda=A\oplus J$ with $A$ a semisimple subalgebra. If the map $J\to J/J^2$ has a splitting as $A$-bimodules, then $\Lambda\cong T_A(J/J^2)$ is a tensor algebra. \qed
\end{prop}

\begin{rem}
The latter condition holds in the following situations.
\begin{enumerate}
\item The ext-quiver of $\Lambda$ is a tree. This includes all hereditary algebras of finite representation type. See \cite[Proposition 10.2]{DR1}.
\item More generally, if there are no indecomposable projective modules $P,P'$ with $\rad(P,P')\neq0$ as well as $\rad^n(P,P')\neq0$ for some $n\geq2$. See \cite[Section 5]{DR3}.
\item The $k$-algebra $A$ is separable \cite{Wedderburn}.
\end{enumerate}

On the other hand, there are finite dimensional hereditary algebras which are not tensor algebras, see \cite{Zaks1} or \cite{DR3}.
\end{rem}

Consider next the fundamental exact sequence
\[ \begin{tikzcd}
0 \arrow[r] & \Omega \arrow[r] & \Lambda\otimes_A\Lambda \arrow[r] & \Lambda \arrow[r] & 0
\end{tikzcd} \]
where the map on the right is just the multiplication map $\lambda\otimes\mu\mapsto\lambda\mu$. The kernel $\Omega$ is sometimes called the $\Lambda$-bimodule of noncommutative 1-forms on $\Lambda$ over $A$. As when constructing the bar resolution, we can identify $\Omega$ with the image of the bimodule homomorphism
\[ \Lambda\otimes_A\Lambda\otimes_A\Lambda \to \Lambda\otimes_A\Lambda, \quad \lambda\otimes\mu\otimes\nu \mapsto \lambda\mu\otimes\nu - \lambda\otimes\mu\nu. \]
Thus $\Omega\leq\Lambda\otimes_A\Lambda$ is generated as a bimodule by the elements $d\lambda\coloneqq 1\otimes\lambda-\lambda\otimes1$.

If $X$ is any $\Lambda$-bimodule, then $\Hom_{\Lambda\hyp\Lambda}(\Omega,X)\cong\Der_A(\Lambda,X)$ is the space of $A$-derivations $\Lambda\to X$, so the $A$-bimodule homomorphisms $d\colon\Lambda\to X$ satisfying $d(\lambda\mu)=\lambda d(\mu)+d(\lambda)\mu$, equivalently those derivations satisfying $d(a)=0$ for all $a\in A$ (cf. \cite[Section 10]{Ginzburg}, though the relative case is not considered there).

\begin{exa}
If $\Lambda=T_A(M)$ is a finite dimensional hereditary algebra, then \intref{Corollary}{co:standard-proj-tensor} tells us that $\Omega=\Lambda\otimes_AM\otimes_A\Lambda$ is a projective bimodule, and hence that $\Der_A(\Lambda,X)\cong\Hom_{A\hyp A}(M,X)$.
\end{exa}

\begin{prop}\label{prop:omega}
Every right $\Lambda$-module $X$ admits a standard projective resolution
\[ \begin{tikzcd}
0 \arrow[r] & X\otimes_\Lambda\Omega \arrow[r] & X\otimes_A\Lambda \arrow[r] & X \arrow[r] & 0
\end{tikzcd} \]
and a standard injective coresolution
\[ \begin{tikzcd}
0 \arrow[r] & X \arrow[r] & \Hom_A(\Lambda,X) \arrow[r] & \Hom_\Lambda(\Omega,X) \arrow[r] & 0.
\end{tikzcd} \]
\end{prop}

\begin{proof}
We tensor the fundamental exact sequence with $X$ to obtain the exact sequence
\[ \begin{tikzcd}
0 \arrow[r] & X\otimes_\Lambda\Omega \arrow[r] & X\otimes_A\Lambda \arrow[r] & X \arrow[r] & 0
\end{tikzcd} \]
with the map $X\otimes_A\Lambda\to X$ being $x\otimes\lambda\mapsto x\lambda$. As $A$ is semisimple we know that $X\otimes_A\Lambda$ is projective, and as $\Lambda$ is hereditary it follows that $X\otimes_\Lambda\Omega$ is also projective.

Similarly, we can apply $\Hom_\Lambda(-,X)$ to the fundamental exact sequence. Here we can identify $\Hom_\Lambda(\Lambda,X)\cong X$ and
\[ \Hom_\Lambda(\Lambda\otimes_A\Lambda,X) \cong \Hom_A(\Lambda,X) \cong X\otimes_AD\Lambda, \]
where $D=\Hom_k(-,k)$ is the usual vector space duality. As $X\otimes_AD\Lambda$ is an injective $\Lambda$-module, so too is $\Hom_\Lambda(\Omega,X)$.
\end{proof}

\begin{rem}
The dual result for left $\Lambda$-modules holds.
\end{rem}

\begin{rem}\label{rem:omega-not-proj-bimod}
Taking $X=A$ shows that $J\cong A\otimes_\Lambda\Omega\cong\Omega/J\Omega$ as $A$-$\Lambda$-bimodules, and hence that $M=J/J^2$ is isomorphic to $\Omega/(J\Omega+\Omega J)$ as $A$-bimodules.

Thus, if $\Omega$ is projective as a $\Lambda$-bimodule, then $\Omega\cong\Lambda\otimes_AM\otimes_A\Lambda$, in which case $J\cong M\otimes_A\Lambda$ as $A$-$\Lambda$-bimodules. Using that $\Lambda=A\oplus J$ we thus obtain a section $M\rightarrowtail J$, showing that $\Lambda=T_A(M)$ is a tensor algebra. 
\end{rem}

\subsection{Ext quiver, Euler form and representation type}\label{sec:ext-quiver-hered}

Let $\Lambda$ be an hereditary algebra in one of our classes, so either finite dimensional or else a locally finite tensor algebra. In both cases we can write $\Lambda=A\oplus\Lambda_+$ with $A$ a (maximal) semisimple subalgebra, and $\Lambda_+$ either the Jacobson radical in the first case, or the graded radical in the second case. Up to Morita equivalence we may further assume that $A$ is basic, so $A=A_1\times\cdots\times A_n$ is a product of division algebras, with corresponding idempotent decomposition $1=e_1+\cdots+e_n$.

Inside the category $\fd\Lambda$ of finite dimensional $\Lambda$-modules we have the Serre subcategory $\nilp\Lambda$ generated by the simples $S_i\coloneqq e_iA$, equivalently the subcategory of modules on which $\Lambda_+$ acts nilpotently. Of course, if $\Lambda$ is finite dimensional, then $\Lambda_+$ is itself nilpotent and so $\nilp\Lambda=\fd\Lambda$; for locally finite tensor algebras, though, the inclusion $\nilp\Lambda\subseteq\fd\Lambda$ may be strict.

We also have the restriction of scalars functor along $A\rightarrowtail\Lambda$, and hence the map of Grothendieck groups $K_0(\fd\Lambda)\to K_0(\modcat A)\cong\bZ^n$. Note that the restriction to $\nilp\Lambda$ yields a bijection $K_0(\nilp\Lambda)\iso K_0(\modcat A)$.

\begin{prop}\label{prop:ext-quiver-lambda}
The Euler form on $\fd\Lambda$
\[ \langle X,Y\rangle \coloneqq \dim\Hom_\Lambda(X,Y) - \dim\Ext^1_\Lambda(X,Y) \]
can be computed using the bilinear form on $K_0(A)$ given by
\[ \langle S_i,S_j\rangle = \delta_{ij}\dim A_i - \dim e_iMe_j, \quad\textrm{where }M\coloneqq\Lambda_+/\Lambda_+^2. \]
\end{prop}

\begin{proof}
We can compute the Euler form using the standard projective resolution of $X$,
\[ \begin{tikzcd}
0 \arrow[r] & X\otimes_\Lambda\Omega \arrow[r] & X\otimes_A\Lambda \arrow[r] & X \arrow[r] & 0
\end{tikzcd} \]
so that
\[ \langle X,Y\rangle = \dim\Hom_A(X,Y) - \dim\Hom_\Lambda(X\otimes_\Lambda\Omega,Y). \]
If $\Lambda=T_A(M)$ is a tensor algebra, then $\Omega=\Lambda\otimes_AM\otimes_A\Lambda$, and so
\[ \Hom_\Lambda(X\otimes_\Lambda\Omega,Y) \cong \Hom_A(X\otimes_AM,Y) \]
as required. If instead $\Lambda$ is finite dimensional, then every finite dimensional module is nilpotent, and as the Euler form is bilinear, we only need know its values on simple modules. In this case we have the projective resolution $0\to\Lambda_+\to\Lambda\to A\to 0$, yielding
\[ \Ext^1_{\,\hyp\Lambda}(A,A) \iso \Hom_{\,\hyp\Lambda}\Lambda(\Lambda_+,A) \cong \Hom_{\,\hyp A}(M,A) \cong DM. \]
We now note that
\[ \Ext^1_{\,\hyp\Lambda}(S_i,S_j) \iso \Hom_{\,\hyp A}(e_iM,e_jA) \cong e_j\cdot DM\cdot e_i \cong D(e_iMe_j). \qedhere \]
\end{proof}

\begin{rem}
Suppose $\Lambda$ is finite dimensional. Comparing the two computations for the Euler form reveals that
\[ \dim\Hom_\Lambda(X\otimes_\Lambda\Omega,Y) = \dim\Hom_\Lambda(X\otimes_AM\otimes_A\Lambda,Y), \]
so by a result of Auslander (see \cite{Bongartz}) we know that, for each $X$, there is an isomorphism of right $\Lambda$-modules $X\otimes_\Lambda\Omega\cong X\otimes_AM\otimes_A\Lambda$.

This isomorphism, however, is in general not natural in $X$, since otherwise it would yield, for $X=A$, an isomorphism of $A$-bimodules $\Lambda_+\cong M\otimes_A\Lambda$, which we saw in \intref{Remark}{rem:omega-not-proj-bimod} happens if and only if $\Lambda=T_A(M)$ is a tensor algebra.
\end{rem}

The proposition above used the radical-squared zero (trivial extension) algebra
\[ \Lambda/\Lambda^2_+ = A\ltimes M \cong A\ltimes D\Ext^1_\Lambda(A,A). \]
Also, we can represent the Euler form on $K_0(A)$ by the matrix $\mat D-\mat R$, where $\mat D=\mathrm{diag}(d_i)$ and $\mat R=(r_{ij})$ are given by $d_i=\dim A_i$ and $r_{ij}=\dim e_iMe_j$. It is common to visualise this as a valued quiver, having vertices $i$ corresponding to the simples $S_i$, and an arrow $i\to j$ with valuation $(r_{ij}/d_i,r_{ij}/d_j)$ whenever $r_{ij}\neq0$. This is precisely the ext-quiver of $\Lambda/\Lambda_+^2$, or equivalently of $\nilp\Lambda$. Observe that the algebra $\Lambda$ is indecomposable (so not a product of algebras) if and only if the valued quiver is connected.

Note however that the numbers $d_i$ are repressed in this description, so it is sometimes appropriate to label the vertex $i$ with the number $d_i$. One can even be more precise and attach the data $A_i$ and $\Ext^1_\Lambda(S_i,S_j)$ (called moduluation by Dlab and Ringel), rather than just the numerical information.

The corresponding symmetric bilinear form $(x,y)\coloneqq\langle x,y\rangle+\langle y,x\rangle$ is represented by the matrix $2\mat D-\mat B$, where $\mat B=(b_{ij})=\mat R+\mat R^t$, and we have the associated symmetrisable matrix $\mat C=2-\bar{\mat B}$, where $\bar{\mat B}=\mat D^{-1}\mat B$. In this case we visualise $\mat C$ as a valued graph, with an edge connecting $i$ and $j$, and having valuation $(b_{ij}/d_i,b_{ji}/d_j)$, whenever $b_{ij}\neq0$. Provided the graph has no `vertex loops', so $e_iMe_i=0$ for all $i$, the matrix $\mat C$ is a symmetrisable (generalised) Cartan matrix; in general it is a symmetrisable Borcherds matrix.

We say that $\Lambda$ is of Dynkin type if $\mat C$ is a Cartan matrix of Dynkin type, equivalently if its valued graph is a Dynkin diagram. Gabriel's Theorem tells us this is equivalent to $\Lambda$ being representation-finite  (see \cite{DR2}). We say that $\Lambda$ is of affine type if the valued graph is an affine, or extended Dynkin, diagram.

\section{The Auslander--Reiten translate}

In order to define the preprojective algebra we need to introduce the Auslander--Reiten translate $\tau$ of a finite dimensional hereditary algebra $\Lambda$, together with the corresponding Auslander--Reiten formula. In particular, we will require that the Auslander--Reiten formula is invariant under $\tau$, a result which we could not find in the literature.

We recall that the translate and its inverse can be defined on the category $\modcat\Lambda$ as
\[ \tau X \coloneqq D\Ext^1_\Lambda(X,\Lambda) \quad\textrm{and}\quad \tau^-X \coloneqq \Ext^1_\Lambda(DX,\Lambda). \]
Moreover, these functors form an adjoint pair $(\tau^-,\tau)$, and induce natural isomorphisms of bifunctors
\[ D\Hom_\Lambda(Y,\tau X) \cong \Ext^1_\Lambda(X,Y) \cong D\Hom_\Lambda(\tau^-Y,X). \]
Our approach will be to construct an exact sequence of $\Lambda$-bimodules
\[ \begin{tikzcd}
\Lambda\otimes_A\Lambda \arrow[r] & \mho \arrow[r] & \Pi_1 \arrow[r] & 0
\end{tikzcd} \]
where $\Lambda=A\oplus J$, such that, for all right $\Lambda$-modules $X$ and left $\Lambda$-modules $Y$, their inverse Auslander--Reiten translates are given by
\[  \tau^-X \coloneqq X\otimes_\Lambda\Pi_1 \quad\textrm{and}\quad \tau^-Y \coloneqq \Pi_1\otimes_\Lambda Y. \]

We shall need the following well-known lemma.

\begin{lem}\label{lem:standard-isos}
Fix a pair of rings $\Lambda$ and $\Gamma$, and modules ${}_\Gamma Y$ and ${}_\Gamma M_\Lambda$.
\begin{enumerate}
\item For a left module ${}_\Lambda X$ there is a natural homomorphism
\[ \Hom_\Gamma(Y,M)\otimes_\Lambda X \to \Hom_\Gamma(Y,M\otimes_\Lambda X), \]
which is an isomorphism whenever $X$ or $Y$ is finitely generated projective.
\item For a right module $X_\Lambda$ there is a natural homomorphism
\[ X\otimes_\Lambda\Hom_\Gamma(M,Y) \to \Hom_\Gamma(\Hom_\Lambda(X,M),Y), \]
which is an isomorphism provided either $X$ is finitely generated projective, or else $X$ is finitely presented and $Y$ is injective.
\end{enumerate}
\end{lem}

\begin{proof}
That we have natural homomorphisms in both cases is clear, so to prove that we have an isomorphism for a finitely generated projective, it is enough to show it for the regular module. For the final case, when $X$ is finitely presented and $Y$ is injective, we know that $\Hom_\Gamma(-,Y)$ is exact, so we can choose a presentation of $X$ and use that right exactness.
\end{proof}

\subsection{Symmetric algebras}

We begin by covering some basic results concerning duality for semisimple algebras. We recall that a (finite dimensional) $k$-algebra $A$ is symmetric provided there is an isomorphism of $A$-bimodules between $A$ and its $k$-dual $DA\coloneqq\Hom_k(A,k)$. Clearly every $A$-bimodule map $A\to DA$ is determined by the image $\sigma\in DA$ of $1\in A$, in which case we have an isomorphism of $A$-bimodules if and only if $\sigma$ is a symmetrising element, meaning that $\sigma(ab)=\sigma(ba)$ and $\sigma(a-)=0$ implies $a=0$.

For convenience we prove the following standard lemma.

\begin{lem}
Every finite dimensional semisimple algebra is symmetric.
\end{lem}

\begin{proof}
If $A$ is symmetric with symmetrising element $\sigma$, then the matrix ring $M_n(A)$ is also symmetric with symmetrising element $(a_{ij})\mapsto\sum\sigma(a_{ii})$. If $B$ is another symmetric algebra, with symmetrising element $\sigma'$, then $A\times B$ is again symmetric, with symmetrising element $(a,b)\mapsto\sigma(a)+\sigma'(b)$. By the Artin--Wedderburn Theorem it is thus enough to show the result when $A$ is a division algebra.

Let $K$ be the centre of $A$ and $L/K$ a splitting field extension, so that $L\otimes_KA\cong M_n(L)$. Writing $C(A)$ for the $K$-subspace of $A$ spanned by all commutators $[a,b]=ab-ba$, we clearly have $L\otimes_KC(A)\subseteq C(L\otimes_KA)$. The latter is a proper subspace of $M_n(L)$, since not all matrices have trace zero. We can now take as symmetrising element any non-zero $\sigma\in DA$ vanishing identically on $C(A)$.
\end{proof}

For the remainder of the paper we shall fix a symmetrising element $\sigma\in DA$. Also, in this section, all unadorned homomorphisms and tensor products will be over $A$, and $k$ will act centrally on all $A$-bimodules.

\begin{lem}\label{lem:A-homs}
Let $X$ and $Y$ be right $A$-modules with $X$ finite dimensional. Then there are isomorphisms
\[ \Hom(X,Y) \leftiso Y\otimes\Hom(X,A) \iso Y\otimes DX \]
sending an element $y\otimes f$ in $Y\otimes\Hom(X,A)$ to the map $x\mapsto yf(x)$ in $\Hom(X,Y)$, and to the element $y\otimes\sigma f$ in $Y\otimes DX$.
\end{lem}

\begin{proof}
The isomorphism on the left follows since $X$ is a finitely generated projective $A$-module. The map on the right is induced by the isomorphism $A\cong DA$.
\end{proof}

We write $\xi\mapsto\!\rmod\xi$ for the inverse of the isomorphism $\Hom(X,A)\iso DX$, $f\mapsto\sigma f$. Thus $\rmod\xi\!(xa)=\rmod\xi\!(x)a$ and $\sigma(\!\rmod\xi\!(x))=\xi(x)$. Note also that this is an isomorphism of left $A$-modules, so $\rmod{a\xi}=a\rmod\xi$, which is the map $x\mapsto a\rmod\xi(x)$.

Taking $X=Y$ in the lemma yields the special case $\End(X)\cong X\otimes DX$. Let $\sum x_i\otimes\xi_i$ correspond to the identity on $X$, so that $x=\sum x_i\rmod{\xi_i}(x)$ for all $x\in X$. For an arbitrary module $Y$ we can then write the isomorphism from the lemma as
\[ \Hom(X,Y) \iso Y\otimes DX, \quad f\mapsto\sum f(x_i)\otimes\xi_i. \]
For, we just need to check that $\sum f(x_i)\rmod{\xi_i}=f$, which follows from
\[ \sum f(x_i)\rmod{\xi_i}(x) = \sum f(x_i\rmod{\xi_i}(x)) = f(x). \]

\begin{rem}\label{rem:alt-actions}
\begin{enumerate}
\item There is an analogous result for left modules. In this case we have an isomorphism
\[ \Hom(X',A)\iso DX', \quad f \mapsto \sigma f, \]
whose inverse we write as $\xi\mapsto\!\lmod\xi$, and for an arbitrary left module $Y$ we have the isomorphism $\Hom(X',Y')\cong DX'\otimes Y'$.
\item If we take a finite dimensional right module $X$ and a left module $Y'$, then we obtain $\Hom(DX,Y')\cong X\otimes Y'$. In this case $x\otimes y'$ is sent to the morphism $\xi\mapsto\!\rmod\xi\!(x)y'$.
\end{enumerate}
\end{rem}

\begin{lem}\label{lem:A-homs-dual}
Let $X$ and $Y$ be finite dimensional right $A$-modules. Then the composite $\Hom(X,Y)\iso Y\otimes DX\iso\Hom(DY,DX)$ sends $f$ to $Df$.

In particular, the same element of $X\otimes DX$ corresponds to both the identity in $\End(X_A)$ and the identity in $\End({}_ADX)$.
\end{lem}

Explicitly, if $\sum x_i\otimes\xi_i$ corresponds to the identity on $X$, then also $\xi=\sum_i\!\rmod\xi\!(x_i)\xi_i$ for all $\xi\in DX$.

\begin{proof}
Take $y\otimes\xi\in Y\otimes DX$. This is sent to $f\colon x\mapsto y\rmod\xi\!(x)$ in $\Hom(X,Y)$, and to $g\colon\eta\mapsto\rmod\eta(y)\xi$ in $\Hom(DY,DX)$. Now
\[ (Df)(\eta)(x) = \eta(f(x)) = \eta(y\rmod\xi\!(x)) = \sigma(\rmod\eta(y)\rmod\xi\!(x)), \]
whereas
\[ g(\eta)(x) = (\rmod\eta(y)\xi)(x) = \xi(x\rmod\eta(y)) = \sigma(\rmod\xi\!(x)\rmod\eta(y)). \]
As $\sigma$ is symmetric we see that $Df=g$ as required. The final statement follows since $D(\id_X)=\id_{DX}$.
\end{proof}

Now assume that we have a finite dimensional $A$-bimodule $M$ and a morphism $X\otimes_AM\to X$ of right $A$-modules, written $x\otimes m\mapsto xm$. We then get the induced map $M\otimes DX\to DX$ such that $(m\xi)(x)\coloneqq\xi(xm)$.

\begin{lem}\label{lem:rmod-action}
Let $\sum m_i\otimes\mu_i$ correspond to the identity on $M_A$, so $m=\sum m_i\rmod{\mu_i}(m)$ for all $m\in M$. Given $x\in X$, the map of left $A$-modules $M\otimes DX\to A$, $m\otimes\xi\mapsto\rmod{m\xi}(x)$ corresponds under the isomorphism $\Hom(M\otimes DX,A)\cong X\otimes DM$ to the element $\sum xm_i\otimes\mu_i$.
\end{lem}

\begin{proof}
We begin by observing that
\[\Hom(M\otimes DX,A) \cong \Hom(DX,\Hom_{A\hyp}(M,A)) \cong \Hom(DX,DM) \cong X\otimes DM. \]
Now $f\colon m\otimes\xi\mapsto\rmod{m\xi}(x)$ and $\sum xm_i\otimes\mu_i$ both yield maps of left $A$-modules $M\otimes DX\to A$, so they correspond if and only if they agree after composing with $\sigma$.

Clearly
\[ \sigma f(m\otimes\xi) = \sigma(\rmod{m\xi}(x)) = (m\xi)(x) = \xi(xm). \]
On the other hand, $\sum xm_i\otimes\mu_i$ corresponds first to the map $DX\to DM$, $\xi\mapsto\sum\!\rmod\xi\!(xm_i)\mu_i$, and then to the morphism
\[ g\colon M\otimes DX \to A, \quad m\otimes\xi\mapsto\sum\lmod\mu_i(m\rmod\xi\!(xm_i)). \]
Now, using $A$-linearity, and the symmetry of $\sigma$, we have
\[ \mu(m\rmod\xi\!(x)) = \sigma\big(\rmod\mu(m)\!\rmod\xi\!(x)\big) = \xi(x\rmod\mu(m)). \]
Thus
\[ \sigma g(m\otimes\xi) = \sum\mu_i(m\rmod\xi\!(xm_i)) = \sum\xi(xm_i\rmod{\mu_i}(m)) = \xi(xm). \]
The result follows.
\end{proof}

\subsection{The Auslander--Reiten translate}\label{sec:ART}

Let $\Lambda=A\oplus J$ be a finite dimensional hereditary algebra. We wish to construct an exact sequence of finite dimensional $\Lambda$-bimodules
\[ \begin{tikzcd}\label{eq:Pi_1}
\Lambda\otimes_A\Lambda \arrow[r] & \mho \arrow[r] & \Pi_1 \arrow[r] & 0 \tag{$\dagger$}
\end{tikzcd} \]
such that, for all right $\Lambda$-modules $X$ and left $\Lambda$-modules $Y$, their inverse Auslander--Reiten translates are given by
\[ \tau^-X \coloneqq X\otimes_\Lambda\Pi_1 \quad\textrm{and}\quad \tau^-Y \coloneqq \Pi_1\otimes_\Lambda Y. \]

Earlier we defined $\Omega$ as the kernel of the multiplication map $\Lambda\otimes_A\Lambda\to\Lambda$. We can therefore define the $\Lambda$-bimodule
\[ \mho \coloneqq D(D\Lambda\otimes_\Lambda\Omega\otimes_\Lambda D\Lambda) \]
as well as the bimodule homomorphism $\Lambda\otimes_A\Lambda \to \mho$ sending $\alpha\otimes\beta$ to the linear functional
\[ \theta\otimes d\lambda\otimes\phi \mapsto \theta(\lmod{\phi\alpha}(\lambda)\beta) - \phi(\alpha\rmod{\beta\theta}(\lambda)). \]
In particular, the image of $1\otimes1$ is the element $\bar c$ having action
\[ \bar c \colon \theta\otimes d\lambda\otimes\phi \mapsto \theta(\lmod\phi(\lambda)) - \phi(\rmod\theta\!(\lambda)). \]

\begin{lem}\label{lem:bar-c-duality}
\begin{enumerate}
\item Let $X$ be a finite dimensional right $\Lambda$-module, and
\[ \Omega\otimes_\Lambda DX \to \Lambda\otimes_ADX, \quad d\lambda\otimes\xi \mapsto 1\otimes\lambda\xi - \lambda\otimes\xi, \]
the left hand map in the standard projective resolution of $DX$. Applying $\Hom_{\Lambda\hyp}(-,\Lambda)$ yields, under the usual natural isomorphisms, the morphism
\[ X\otimes_A\Lambda \to X\otimes_\Lambda\mho, \quad x\otimes1 \mapsto x\otimes\bar c. \]
\item Dually, let $Y$ be a finite dimensional left $\Lambda$-module, and
\[ DY\otimes_\Lambda\Omega \to DY\otimes_A\Lambda, \quad \eta\otimes d\lambda \mapsto \eta\otimes\lambda - \eta\lambda\otimes1, \]
the left hand map in the standard projective resolution of $DY$. Applying $\Hom_{\,\hyp\Lambda}(-,\Lambda)$ yields, under the usual natural isomorphisms, the morphism
\[ \Lambda\otimes_AY \to \mho\otimes_\Lambda Y, \quad 1\otimes y \mapsto \bar c\otimes y. \]
\end{enumerate}
\end{lem}

\begin{proof}
(1) Applying $\Hom_{\Lambda\hyp}(-,\Lambda)$ to the projective resolution of $DX$ and using the isomorphism $\Hom_\Lambda(\Lambda\otimes_ADX,\Lambda)\cong X\otimes_A\Lambda$ yields the map of right $\Lambda$-modules
\[ X\otimes_A\Lambda \to \Hom_\Lambda(\Omega\otimes_\Lambda DX,\Lambda) \]
sending $x\otimes1$ to the morphism
\[ d\lambda\otimes\xi \mapsto \rmod{\lambda\xi}(x)-\lambda\rmod\xi(x). \]

On the other hand, the action of $\bar c$ can alternatively be expressed as
\[ \theta\otimes d\lambda\otimes\phi \mapsto \theta(\lmod\phi(\lambda)) - \theta(\lambda\lmod\phi(1)), \]
so under the usual tensor-hom adjunction $\mho\cong\Hom_{\Lambda\hyp}(\Omega\otimes_\Lambda D\Lambda,\Lambda)$ the element $\bar c$ corresponds to the left $\Lambda$-homomorphism
\[ \Psi \colon \Omega\otimes_\Lambda D\Lambda \to \Lambda, \quad d\lambda\otimes\phi \mapsto \lmod\phi(\lambda)-\lambda\lmod\phi(1). \]

We recall from \intref{Proposition}{prop:omega} that $\Hom_{\Lambda\hyp}(\Omega,\Lambda)$ is an injective left $\Lambda$-module. We can therefore apply \intref{Lemma}{lem:standard-isos} to obtain a natural isomorphism
\[ X\otimes_\Lambda\Hom_{\Lambda\hyp}(D\Lambda,\Hom_{\Lambda\hyp}(\Omega,\Lambda)) \cong
\Hom_{\Lambda\hyp}(\Hom_{\,\hyp\Lambda}(X,D\Lambda),\Hom_{\Lambda\hyp}(\Omega,\Lambda)). \]
Now $\Hom_\Lambda(X,D\Lambda)\cong DX$, so we have the natural isomorphisms
\[ X\otimes_\Lambda\mho \iso X\otimes_\Lambda \Hom_{\Lambda\hyp}(\Omega\otimes_\Lambda D\Lambda,\Lambda) \iso \Hom_\Lambda(\Omega\otimes_\Lambda DX,\Lambda). \]
In particular, this sends $x\otimes\bar c$ to the morphism
\[ d\lambda\otimes\xi \mapsto \Psi(d\lambda\otimes\phi)=\lmod\phi(\lambda)-\lambda\lmod\phi(1), \quad\textrm{where }
\phi(\mu) \coloneqq \xi(x\mu). \]
Finally, we check that $\lmod\phi(\lambda)=\rmod{\lambda\xi}(x)$ for all $\lambda$. For, given any $a\in A$ we have
\[ \sigma(a\lmod\phi(\lambda)) = \phi(a\lambda) = \xi(xa\lambda) =(\lambda\xi)(xa) = \sigma(\rmod{\lambda\xi}(x)a), \]
and so the claim follows from the non-degeneracy of $\sigma$.

(2) The proof for left modules is analogous. In this case we have
\[ \Lambda\otimes_AY \to \Hom_\Lambda(DY\otimes_\Lambda\Omega,\Lambda) \]
sending $1\otimes y$ to the map
\[ \eta\otimes d\lambda \mapsto \lmod\eta(y)\lambda - \lmod{\eta\lambda}(y). \]
Next, the isomorphism $\mho\iso\Hom_{\,\hyp\Lambda}(D\Lambda\otimes_\Lambda\Omega,\Lambda)$ sends $\bar c$ to the map
\[ \theta\otimes d\lambda \mapsto \rmod\theta(1)\lambda - \rmod\theta(\lambda). \]
Thus the isomorphism
\[ \mho\otimes_\Lambda Y \cong \Hom_\Lambda(DY\otimes_\Lambda\Omega,\Lambda) \]
sends $\bar c\otimes y$ to the morphism
\[ \eta\otimes d\lambda \mapsto \rmod\theta(1)\lambda - \rmod\theta(\lambda), \quad\textrm{where } \theta(\mu) \coloneqq \eta(\mu y). \]
Finally, $\rmod\theta(\lambda)=\lmod{\eta\lambda}(y)$ for all $\lambda$.
\end{proof}

\begin{prop}\label{prop:AR-translate}
Let $X$ be a finite dimensional right $\Lambda$-module. Then $\tau^-X\cong X\otimes_\Lambda\Pi_1$ and $\tau X\cong\Hom_\Lambda(\Pi_1,X)$. Analogously for left $\Lambda$-modules.
\end{prop}

\begin{proof}
The module $\tau^-X=\Ext^1_\Lambda(DX,\Lambda)$ can be computed by applying $\Hom_\Lambda(-,\Lambda)$ to the standard projective presentation of $DX$. Explicitly, it is the cokernel of the morphism considered in the previous lemma, which we showed is isomorphic to $X\otimes_\Lambda\Pi_1$. The corresponding right adjoints are necessarily also isomorphic, so $\tau X\cong\Hom_\Lambda(\Pi_1,X)$.
\end{proof}

In general we define $\tau^-X\coloneqq X\otimes_\Lambda\Pi_1$ and $\tau X\coloneqq\Hom_\Lambda(\Pi_1,X)$ for all right $\Lambda$-modules $X$.

\begin{cor}\label{cor:dual-AR}
\pushQED{\qed}
For an arbitrary right $\Lambda$-module $X$ we have
\[ D(\tau^-X) \cong \Hom_\Lambda(\Pi_1,DX) \cong \tau(DX). \qedhere \]
\end{cor}

\subsection{The Auslander--Reiten formula}

Our aim here is to prove the Auslander--Reiten formula
\[ D\Hom_\Lambda(Y,\tau X) \cong \Ext^1_\Lambda(X,Y) \cong D\Hom_\Lambda(\tau^-Y,X). \]

We begin by noting that for $X\in\modcat\Lambda$ there is a linear functional $t_X$ on
\[ X\otimes_\Lambda DX \cong D\End_\Lambda(X) \]
sending $x\otimes\xi$ to $\xi(x)$, or equivalently evaluating $\theta\in D\End_\Lambda(X)$ at the identity $\id_X$. That these descriptions agree follows from the fact that the isomorphism $X\otimes_\Lambda DX\iso D\End_\Lambda(X)$ sends $x\otimes\xi$ to the map $f\mapsto\xi(f(x))$.

We remark that each $h\colon X\to Y$ yields
\[ Dh_\ast \colon D\Hom_\Lambda(X,Y)\to D\End_\Lambda(X) \]
and
\[ Dh^\ast \colon D\Hom_\Lambda(Y,Y) \to D\Hom_\Lambda(X,Y) \]
with the property that 
\[ t_X(Dh_\ast(\theta)) = \alpha(h) = t_Y(Dh^\ast(\theta)) \quad\textrm{for all }\theta\in D\Hom(X,Y). \]

\begin{lem}
For $X,P\in\modcat\Lambda$ with $P$ projective there is a natural isomorphism
\[ \Hom_\Lambda(X,D\Hom_\Lambda(P,\Lambda)) \iso D\Hom_\Lambda(P,X), \quad f \mapsto \tilde f, \]
in which case, for all $h\colon P\to X$ we have
\[ t_P(\widetilde{fh}) = \tilde f(h). \]
\end{lem}

\begin{proof}
By \intref{Lemma}{lem:standard-isos} (1) we have the isomorphism
\[ \alpha \colon X\otimes_\Lambda\Hom_\Lambda(P,\Lambda) \cong \Hom_\Lambda(P,X) \]
sending $x\otimes\pi$ to the map $p\mapsto x\pi(p)$. This in turn yields the isomorphisms
\[ \Hom_\Lambda(X,D\Hom_\Lambda(P,\Lambda)) \cong D(X\otimes_\Lambda\Hom_\Lambda(P,\Lambda)) \cong D\Hom_\Lambda(P,X). \]
In particular, given $f\colon X\to D\Hom_\Lambda(P,\Lambda)$, we have $\tilde f(\alpha(x\otimes\pi))=f(x)(\pi)$.

We then see that for $Dh_\ast\colon D\Hom_\Lambda(P,X)\to D\End_\Lambda(P)$ we have $\widetilde{fh}=Dh_\ast(\tilde f)$, since both send $\alpha(p\otimes\pi)$ to $f(h(p))(\pi)$. Thus $t_P(Dh_\ast(\tilde f))=\tilde f(h)$ by the earlier remark.

Alternatively, let $\sum_ip_i\otimes\pi_i\in P\otimes_\Lambda D\Hom_\Lambda(P,\Lambda)$ correspond to the identity on $P$, so that $p=\sum_ip_i\pi_i(p)$ for all $p\in P$. Then $h\colon P\to X$ corresponds to $\sum_ih(p_i)\otimes\pi_i$, and so
\[ \tilde f(h) = \sum_i f(h(p_i))(\pi_i) = \widetilde{fg}(\id_P) = t_P(\widetilde{fg}). \qedhere \]
\end{proof}

\begin{thm}\label{thm:ARF}
There is a bifunctorial perfect pairing, called the Auslander--Reiten pairing,
\[ \{-,-\} \colon \Ext^1_\Lambda(X,Y) \times \Hom_\Lambda(Y,\tau X) \to k. \]
In particular, $\{\eta,g\}$ can be computed by choosing a projective resolution of $X$
\[ \begin{tikzcd}
0 \arrow[r] & P \arrow[r] & P' \arrow[r] & X \arrow[r] & 0
\end{tikzcd} \]
writing $\eta$ as the pushout along some $f\colon P\to Y$, and taking $t_P(\widetilde{igf})$, where $i\colon\tau X\rightarrowtail D\Hom_\Lambda(P_1,\Lambda)$ is the inclusion.
\end{thm}

\begin{proof}
Choosing such a projective resolution of $X$ we can use the lemma to obtain a commutative square
\[ \begin{tikzcd}
\Hom_\Lambda(Y,D\Hom_\Lambda(P,\Lambda)) \arrow[r,"\sim"] \arrow[d] & D\Hom_\Lambda(P,Y) \arrow[d]\\
\Hom_\Lambda(Y,D\Hom_\Lambda(P',\Lambda)) \arrow[r,"\sim"] & D\Hom_\Lambda(P',Y)
\end{tikzcd} \]
We therefore obtain a natural isomorphism between the kernels of the vertical maps
\[ \Hom_\Lambda(Y,\tau X) \cong D\Ext^1_\Lambda(X,Y), \quad \tau X = D\Ext^1_\Lambda(X,\Lambda). \]
Explicitly, given $\eta\in\Ext^1_\Lambda(X,Y)$, we write this as the pushout along some $f\colon P\to Y$ of the projective resolution of $X$, in which case $\{\eta,g\}$ equals the natural pairing between $f$ and $\widetilde{ig}$, which by the previous lemma is $t_P(\widetilde{igf})$.

We now show that this construction is independent of the chosen projective presentation of $X$. We know that we can always relate two such projective presentations, yielding a pushout diagram
\[ \begin{tikzcd}
0 \arrow[r] & Q \arrow[r] \arrow[d,"p"] & Q' \arrow[r] \arrow[d] & X \arrow[r] \arrow[d,equal] & 0\\
0 \arrow[r] & P \arrow[r] & P' \arrow[r] & X \arrow[r] & 0
\end{tikzcd} \]
in which case $\eta$ is also the pushout along $fp$. We also have the inclusion $j\colon\tau X\rightarrowtail D\Hom_\Lambda(Q,\Lambda)$ such that $i=(Dp^\ast)j$. Now $jgf\in\Hom_\Lambda(P,D\Hom_\Lambda(Q,\Lambda))\cong D\Hom(Q,P)$, so, with a slight abuse of notation,
\[ t_P(pjgf) = (jgf)(p) = t_Q(jgfp). \qedhere \]
\end{proof}

\section{Preprojective algebras}

The preprojective algebra $\Pi$ of a finite dimensional hereditary algebra $\Lambda$ was defined in \cite{BGL} to be the tensor algebra $T_\Lambda(\tau^-\Lambda)$, whereas for the path algebra of a quiver $\Lambda=kQ$, its preprojective algebra is given as $\Pi=k\bar Q/\langle c\rangle$, where $\bar Q$ is the doubled quiver, formed by adjoining an arrow $a^\ast\colon j\to i$ for each arrow $a\colon i\to j$ of $Q$, and $c=\sum[a,a^\ast]$.

For acyclic quivers a similar notion was introduced earlier in \cite{GP} using the Coxeter functor rather than the Auslander--Reiten translate. Both algebras have the property that, when viewed as a right $\Lambda$-module, they decompose as the direct sum of all indecomposable preprojective $\Lambda$-modules, hence the name. The definition by Baer, Geigle and Lenzing, however, has proven to be the more fruitful, and occurs naturally in many diverse settings.

Our aim in this section is to unify both of these constructions by giving a sequence of $\Lambda$-bimodules
\[ \begin{tikzcd}
\Lambda\otimes_A\Lambda \arrow[r] & \mho \arrow[r] & \Pi_1 \arrow[r] & 0
\end{tikzcd} \]
when $\Lambda$ is one of our hereditary algebras, and defining the preprojective algebra as the tensor algebra $\Pi=T_\Lambda(\Pi_1)$. The above sequence then tells us that $\Pi$ is a quotient of the tensor algebra $T_\Lambda(\mho)$ by the ideal generated by $\bar c$, the image of $1\otimes1$ under the left hand map, and we will furthermore show that $T_\Lambda(\mho)$ is always hereditary. In particular, when $\Lambda=kQ$ is the path algebra of a quiver, then $T_\Lambda(\mho)\cong k\bar Q$ is isomorphic to the path algebra of the doubled quiver.

For convenience we say that the preprojective algebra $\Pi$ has the same `type' as the underlying hereditary algebra $\Lambda$. For example, we will say that $\Pi$ is of Dynkin type if $\Lambda$ is of Dynkin type, equivalently if $\Lambda$ has finite representation type.

Recall that we are interested in two types of hereditary algebras: those which are finite dimensional, and those which are locally finite tensro algebras. In both cases we can write $\Lambda=A\oplus\Lambda_+$ with $A$ a semisimple subalgebra, and $\Lambda_+$ the (Jacobson or graded) radical.

\subsection{Preprojective algebra of a finite dimensional hereditary algebra}

Let $\Lambda=A\oplus\Lambda_+$ be a finite dimensional hereditary algebra. In this case we have defined the $\Lambda$-bimodule $\Pi_1$ in \intref{Section}{sec:ART} via the sequence
\[ \begin{tikzcd}
\Lambda\otimes_A\Lambda \arrow[r] & \mho \arrow[r] & \Pi_1 \arrow[r] & 0
\end{tikzcd} \]
where $\mho=D(D\Lambda\otimes_\Lambda\Omega\otimes_\Lambda D\Lambda)$ and the map on the left sends $1\otimes1$ to
\[ \bar c \colon \theta\otimes d\lambda\otimes\phi \mapsto \theta(\lmod\phi(\lambda)) - \phi(\rmod\theta\!(\lambda)). \]
The preprojective algebra of $\Lambda$ is then $\Pi=T_\Lambda(\Pi_1)$, and the universal property for tensor algebras tells us that $\Pi$ is isomorphic to the quotient of the tensor algebra $T_\Lambda(\mho)$ be the ideal $I=\langle\bar c\rangle$.

Clearly $\Pi=\bigoplus_n\Pi_n$ is a graded algebra with $\Pi_0=\Lambda$, and using \intref{Proposition}{prop:AR-translate} we have a bimodule isomorphism $\Pi_n\cong\tau^{-n}\Lambda$, using the inverse Auslander--Reiten translate from either the category of left or right $\Lambda$-modules.

Our aim in this section is to show that the tensor algebra $T_\Lambda(\mho)$ is in fact hereditary, and to compute a standard projective resolution for each $\Pi$-module.

\begin{lem}\label{lem:mho-proj}
We have that $X\otimes_\Lambda\mho$ is projective as a right $\Lambda$-module for each $X\in\modcat\Lambda$. Analogously $\mho\otimes_\Lambda Y$ is projective for each left $\Lambda$-module $Y$.
\end{lem}

\begin{proof}
Assume first that $X$ is finite dimensional. We saw in the proof of \intref{Proposition}{prop:AR-translate} that there is a natural isomorphism
\[ X\otimes_\Lambda\mho \cong \Hom_\Lambda(\Omega\otimes_\Lambda DX,\Lambda). \]
Now $\Omega\otimes_\Lambda DX$ is a projective left $\Lambda$-module by \intref{Proposition}{prop:omega}, so the right hand side above is a projective right $\Lambda$-module.

This shows that $X\otimes_\Lambda\mho$ is a projective right $\Lambda$-module whenever $X$ is finite dimensional. In general, as $\Lambda$ is finite dimensional, each module $X$ is the directed colimit (union) of its finite dimensional submodules. It follows that $X\otimes_\Lambda\mho$ is a directed colimit of projective modules, so flat, and hence projective as $\Lambda$ is finite dimensional (and so perfect).

The result for left modules is analogous.
\end{proof}

We have algebra inclusions $A\rightarrowtail\Lambda\rightarrowtail T$, so we again have a fundamental exact sequence
\[ \begin{tikzcd}
0 \arrow[r] & \widetilde\Omega \arrow[r] & T\otimes_AT \arrow[r] & T \arrow[r] & 0
\end{tikzcd} \]
with $\widetilde\Omega$ the bimodule of noncommutative 1-forms on $T$ over $A$.

\begin{thm}\label{thm:tensor-mho-hered}
The tensor algebra $T=T_\Lambda(\mho)$ is hereditary, every right module $X$ has a functorial projective presentation
\[ \begin{tikzcd}
0 \arrow[r] & X\otimes_T\widetilde\Omega \arrow[r] & X\otimes_AT \arrow[r] & X \arrow[r] & 0,
\end{tikzcd} \]
and there is a functorial exact sequence
\[ \begin{tikzcd}
0 \arrow[r] & X\otimes_\Lambda\Omega\otimes_\Lambda T \arrow[r] & X\otimes_T\widetilde\Omega \arrow[r] &
X\otimes_\Lambda\mho\otimes_\Lambda T \arrow[r] & 0.
\end{tikzcd} \]
This latter sequence is split exact, though the splitting is not functorial.
\end{thm}

\begin{proof}
As usual we can tensor the fundamental exact sequence for $T$ with $X$ to get
\[ \begin{tikzcd}
0 \arrow[r] & X\otimes_T\widetilde\Omega \arrow[r] & X\otimes_AT \arrow[r] & X \arrow[r] & 0
\end{tikzcd} \]
Next, regarding $X$ as a right $\Lambda$-module we have the short exact sequence
\[ \begin{tikzcd}
0 \arrow[r] & X\otimes_\Lambda\Omega \arrow[r] & X\otimes_A\Lambda \arrow[r] & X \arrow[r] & 0.
\end{tikzcd} \]
We tensor this on the right by $T$, noting that $T=\bigoplus_n\mho^{\otimes n}$ is a projective left $\Lambda$-module. This yields the exact sequence
\[ \begin{tikzcd}
0 \arrow[r] & X\otimes_\Lambda\Omega\otimes_\Lambda T \arrow[r] & X\otimes_AT \arrow[r] & X\otimes_\Lambda T \arrow[r] &0.
\end{tikzcd} \]
On the other hand we have the fundamental exact sequence for $T=T_\Lambda(\mho)$ given in \intref{Proposition}{prop:tensor-alg-seq}, which we may tensor on the left with $X$ to get
\[ \begin{tikzcd}
0 \arrow[r] & X\otimes_\Lambda\mho\otimes_\Lambda T \arrow[r] & X\otimes_\Lambda T \arrow[r] & X \arrow[r] & 0.
\end{tikzcd} \]
Putting these together yields the following exact commutative diagram
\[ \begin{tikzcd}
& X\otimes_\Lambda\Omega\otimes_\Lambda T \arrow[r,equal] \arrow[d,rightarrowtail] & X\otimes_\Lambda\Omega\otimes_\Lambda T \arrow[d,rightarrowtail]\\
0 \arrow[r] & X\otimes_T\widetilde\Omega \arrow[r] \arrow[d,twoheadrightarrow] &
X\otimes_AT \arrow[r] \arrow[d,twoheadrightarrow] & X \arrow[r] \arrow[d,equal] & 0\\
0 \arrow[r] & X\otimes_\Lambda\mho\otimes_\Lambda T \arrow[r] & X\otimes_\Lambda T \arrow[r] & X \arrow[r] & 0
\end{tikzcd} \]
Finally, as both $X\otimes_\Lambda\Omega$ and $X\otimes_\Lambda\mho$ are projective right $\Lambda$-modules, by \intref{Propositions}{prop:omega} and \intref{Lemma}{lem:mho-proj}, the left hand column is split exact, and $X\otimes_T\widetilde\Omega$ is a projective right $T$-module.
\end{proof}

\begin{rem}
Suppose $\Lambda$ is finite dimensional, but not a tensor algebra. We saw in \intref{Remark}{rem:omega-not-proj-bimod} that $\Omega$ is not projective as a $\Lambda$-bimodule. Similarly, noting that $\Omega\cong D(D\Lambda\otimes_\Lambda\mho\otimes_\Lambda D\Lambda)$, we see that $\mho$ is not projective as a bimodule. Thus the $T$-bimodule map $\widetilde\Omega\twoheadrightarrow T\otimes_\Lambda\mho\otimes_\Lambda T$ is split as both left and right $T$-modules, but not as $T$-bimodules.
\end{rem}

Our final result  in this section is taken from \cite[Proposition 3.5]{BGL} and shows that $\Pi$ is invariant under applying APR-tilts.

We will need the following alternative description of the preprojective algebra, as
\[ \Pi = \bigoplus_{n\geq0}\Hom_\Lambda(\Lambda,\tau^{-n}\Lambda), \]
using the standard identification $X\cong\Hom_\Lambda(\Lambda,X)$. The multiplication is now given by $g\cdot f=(\tau^{-m}g)f$ for $f\in\Hom_\Lambda(\Lambda,\tau^{-m}\Lambda)$ and $g\in\Hom_\Lambda(\Lambda,\tau^{-n}\Lambda)$.

\begin{prop}
Let $X\in\modcat\Lambda$ be a preprojective tilting module, with endomorphism algebra $\Lambda'$. If $\Lambda$ is connected and of infinite representation type, then the preprojective algebras of $\Lambda$ and $\Lambda'$ are Morita equivalent.
\end{prop}

\begin{proof}
We sketch the proof given in \cite{BGL}. We write $\Pi$ and $\Pi'$ for the preprojective algebras of $\Lambda$ and $\Lambda'$ respectively, and define $\Pi(X)=\bigoplus_{n\geq0}\Hom_\Lambda(X,\tau^{-n}X)$.

The tilting functor $\Hom_\Lambda(X,-)$ induces an equivalence between the subcategory $\mathcal T=\gen X=\Ker\Ext^1_\Lambda(X,-)$ of $\modcat\Lambda$ and the subcategory $\mathcal Y=\Ker\Tor_1^{\Lambda'}(-,X)$ of $\modcat\Lambda'$, with quasi-inverse the functor $-\otimes_{\Lambda'}X$. In particular $\Hom_\Lambda(X,-)$ sends a (relative) Auslander--Reiten sequence in $\mathcal T$ to an Auslander--Reiten sequence in $\modcat\Lambda'$ with all terms in $\mathcal Y$. (See for example \cite[Theorem 6.5]{HR}, especially part (2) of the proof.) As the Auslander--Reiten sequence $0\to \tau^{-n}X\to E_n\to \tau^{-n-1}X\to 0$ lies in $\mathcal T$, we see by induction that $\tau^{-n}\Lambda'\cong\Hom_\Lambda(X,\tau^{-n}X)$.

It follows that there is an isomorphism of graded algebras $\Pi(X)\cong\Pi'$.

Next, up to multiplicity of direct summands we can obtain $X$ from $\Lambda_\Lambda$ via a sequence of APR-tilts, so let us write $\Lambda_\Lambda=P\oplus S$ and $X=P\oplus\tau^-S$ such that $\Hom_\Lambda(P,S)=0$. Then also $\Hom_\Lambda(\tau^-S,P)=0$, and
\[ \Hom(\Lambda,\tau^{-n}\Lambda) = \begin{bmatrix}
\Hom(P,\tau^{-n}P)&\Hom(S,\tau^{-n}P)\\\Hom(P,\tau^{-n}S)&\Hom(S,\tau^{-n}S)
\end{bmatrix} \]
whereas
\[ \Hom(X,\tau^{-n}X) \cong \begin{bmatrix}
\Hom(P,\tau^{-n}P)&\Hom(S,\tau^{-n+1}P)\\\Hom(P,\tau^{-n-1}S)&\Hom(S,\tau^{-n}S)
\end{bmatrix}. \]
The obvious map $\Pi\to\Pi(X)$ is then an isomorphism of algebras (but is not a map of graded algebras).
\end{proof}

\begin{rem}
As all representation finite hereditary algebras are tensor algebras, the analogous result in this case is covered by \intref{Lemma}{lem:indep-orient}.
\end{rem}

\subsection{Preprojective algebra of an hereditary tensor algebra}

In this section we define the preprojective algebra of an arbitrary locally finite hereditary tensor algebra $\Lambda=T_A(M)$, generalising the definition in the finite dimensional case. Note that the bimodule of differential 1-forms is $\Omega=\Lambda\otimes_AM\otimes_A\Lambda$.

We begin by analysing the case when $\Lambda$ is finite dimensional. We recall that for finite dimensional right $A$-modules $X$ and $Y$ we have the isomorphisms
\[ X\otimes_ADY \cong \Hom_A(Y,X) \cong \Hom_A(Y,D^2X) \cong D(Y\otimes_ADX). \]
Applying this to
\[ \mho = D(D\Lambda\otimes_\Lambda\Omega\otimes_\Lambda D\Lambda) \cong D(D\Lambda\otimes_AM\otimes_AD\Lambda) \]
we obtain that
\[ \mho \cong \Lambda\otimes_AD(D\Lambda\otimes_AM) \cong \Lambda\otimes_ADM\otimes_A\Lambda. \]
The defining sequence for $\Pi_1$ thus becomes
\[ \begin{tikzcd}
\Lambda\otimes_A\Lambda \arrow[r] & \Lambda\otimes_ADM\otimes_A\Lambda \arrow[r] & \Pi_1 \arrow[r] & 0
\end{tikzcd} \]
and we wish to describe the morphism on the left. For this we will need to fix representatives
\[ \End(M_A) \cong M\otimes_ADM, \quad \id_{M_A} \leftrightarrow \rho = \sum m_i\otimes\mu_i \]
and
\[ \End({}_AM) \cong DM\otimes_AM, \quad \id_{{}_AM} \leftrightarrow \ell = \sum \nu_j\otimes n_j. \]
Thus each element of $M$ satisfies both $m=\sum m_i\rmod\mu_i(m)$ and $m=\sum\lmod\nu_j(m)n_j$.

\begin{prop}\label{prop:bar-c-tensor}
Let $\Lambda=T_A(M)$ be a finite dimensional hereditary tensor algebra. Then in the defining sequence
\[ \begin{tikzcd}
\Lambda\otimes_A\Lambda \arrow[r] & \Lambda\otimes_ADM\otimes_A\Lambda \arrow[r] & \Pi_1 \arrow[r] & 0
\end{tikzcd} \]
the map on the left sends $1\otimes1$ to
\[ \bar c \coloneqq \rho\otimes1-1\otimes\ell =  \sum m_i\otimes\mu_i\otimes1 - \sum 1\otimes\nu_j\otimes n_j. \]
\end{prop}

\begin{proof}
Under the isomorphism $\Lambda\otimes_ADM \cong \Hom{\,\hyp\Lambda}(M,\Lambda)$, the element $\sum m_i\otimes\mu_i$ corresponds to the inclusion $\iota\colon M\rightarrowtail\Lambda$. Next, under the isomorphism
\[ \Hom_{\,\hyp\Lambda}(M,\Lambda)\otimes_A\Lambda \cong \Hom_{\,\hyp\Lambda}(D\Lambda,\Hom_{\,\hyp\Lambda}(M,\Lambda))
\cong \Hom_{\,\hyp\Lambda}(D\Lambda\otimes_AM,\Lambda) \]
the element $\iota\otimes1$ corresponds first to the morphism $\theta\mapsto\iota\lmod\theta(1)$, and then to the morphism $\theta\otimes m\mapsto\lmod\theta(1)m$. Finally, this corresponds in $\mho=D(D\Lambda\otimes_AM\otimes_A\Lambda)$ to the map sending $\theta\otimes m\otimes\phi$ to
\[ \phi(\lmod\theta(1)m) = \sigma(\lmod\theta(1)\lmod\phi(m)) = \sigma(\lmod\phi(m)\lmod\theta(1)) = \theta(\lmod\phi(m)). \]

Similarly $\sum1\otimes\nu_j\otimes n_j$ corresponds in $\Hom_{A\hyp}(D\Lambda,\Hom_{A\hyp}(M,\Lambda))$ to the map $\phi\mapsto\rmod\phi(1)\iota$, and so to $m\otimes\phi\mapsto m\rmod\phi(1)$ in $\Hom_{A\hyp}(M\otimes_AD\Lambda,\Lambda)$. This then corresponds in $\mho$ to the map sending $\theta\otimes m\otimes\phi$ to
\[ \theta(m\rmod\phi(1)) = \phi(\rmod\theta(m)). \]

We conclude that $\rho\otimes1-1\otimes\ell$ corresponds to the element $\bar c$, sending $\theta\otimes m\otimes\phi$ to $\theta(\lmod\phi(m))-\phi(\rmod\theta(m))$.
\end{proof}

\begin{thm}
Let $\Lambda=T_A(M)$ be a finite dimensional hereditary tensor algebra. Then $T_\Lambda(\mho)\cong T_A(M\oplus DM)$, under which $\bar c\in\mho$ is sent to
\[ c = \rho-\ell \in (M\otimes_ADM)\oplus(DM\otimes_AM). \]
Thus the preprojective algebra $\Pi$ of $\Lambda$ is isomorphic to $T_A(M\oplus DM)/\langle c\rangle$.
\end{thm}

As a slogan we see that the ideal is generated by the element
\[ c = \id_{M_A} - \id_{{}_AM}. \]

\begin{proof}
We know that $\Pi$ is isomorphic to $T_\Lambda(\mho)/\langle\bar c\rangle$, and we saw above that $\mho\cong\Lambda\otimes_ADM\otimes_A\Lambda$. Using that $\Lambda=T_A(M)$, the universal property for tensor algebras tells us that $T_\Lambda(\mho)\cong T_A(M\oplus DM)$. Moreover, this isomorphism sends the element $\bar c=\rho\otimes1-1\otimes\ell$ in $\Lambda\otimes_ADM\otimes_A\Lambda$ to $c=\rho-\ell$.
\end{proof}

We now observe that, quite generally, $a\rho=\rho a$ in $M\otimes_ADM$ for all $a\in A$, since both elements correspond in $\End_{\,\hyp A}(M)$ to left multiplication by $a$. Similarly $a\ell=\ell a$, and so for any locally finite hereditary tensor algebra $\Lambda=T_A(M)$ we can always define a sequence of bimodules
\[ \begin{tikzcd}
\Lambda\otimes_A\Lambda \arrow[r] & \Lambda\otimes_ADM\otimes_A\Lambda \arrow[r] & \Pi_1 \arrow[r] & 0
\end{tikzcd} \]
such that the left hand map sends $1\otimes1$ to $\bar c=\rho\otimes1-1\otimes\ell$. Now $T_\Lambda(\Lambda\otimes_ADM\otimes_A\Lambda)\cong T_A(M\oplus DM)$, under which $\bar c$ is sent to $c=\rho-\ell$. We may therefore define the preprojective algebra of $\Lambda$ to be
\[ \Pi \coloneqq T_\Lambda(\Pi_1) \cong T_A(M\oplus DM)/\langle c\rangle. \]

The next result shows that the preprojective algebra of an hereditary tensor algebra is independent of the `orientation'.

\begin{lem}\label{lem:indep-orient}
The preprojective algebras associated to the locally finite hereditary tensor algebras $T_A(M\oplus N)$ and $T_A(M\oplus DN)$ are isomorphic.
\end{lem}

\begin{proof}
We take $\rho_1\in M\otimes_ADM$ and $\rho_2\in N\otimes_ADN$ corresponding to the identities on $M_A$ and $N_A$ respectively. Similarly we take $\ell_1\in DM\otimes_AM$ and $\ell_2\in DN\otimes_AN$ corresponding to the identities on ${}_AM$ and ${}_AN$.

For convenience we set $\widetilde M\coloneqq M\oplus DM$, and similarly for $\widetilde N$. Then, if we write $c_1=\rho_1-\ell_1$ and $c_2=\rho_2-\ell_2$, the preprojective algebra of $T_A(M\oplus N)$ is the quotient $T_A(\widetilde M\oplus\widetilde N)/\langle c_1+c_2\rangle$. On the other hand, we know from \intref{Lemma}{lem:A-homs-dual} that $\rho_2$ and $\ell_2$ correspond respectively to the identities on ${}_A(DN)$ and $(DN)_A$. Thus the preprojective algebra of $T_A(M\oplus DN)$ is the quotient $T_A(\widetilde M\oplus\widetilde DN)/\langle c_1-c_2\rangle$.

Finally, the automorphism of $N\oplus DN$ sending $(n,\nu)$ to $(n,-\nu)$ induces an automorphism of $T_A(\widetilde M\oplus\widetilde N)$ exchanging $c_2$ with $-c_2$, and hence induces an isomorphism between the two preprojective algebras.
\end{proof}

\begin{exa}
Let $\Lambda=kQ$ be the path algebra of a quiver. Then $A=k^n$, having basis the vertices of $Q$, and $M$ is a direct sum of one dimensional bimodules, having basis the arrows $a$ of $Q$. Writing $a^\ast$ for the dual basis of $DM$, we see that $\rho=\sum a\otimes a^\ast$ and $\ell=\sum a^\ast\otimes a$, so that $c=\sum[a,a^\ast]$ and we recover the well-known description of the preprojective algebra $\Pi=k\bar Q/\langle\sum[a,a^\ast]\rangle$.
\end{exa}

\begin{lem}
Let $\Lambda$ be an hereditary algebra in one of our classes. Then its preprojective algebra $\Pi$ is finite dimensional if and only if $\Lambda$ is finite dimensional and of finite representation type.
\end{lem}

\begin{proof}
We begin by noting that $\Lambda$ is always a subalgebra of $\Pi$, so we can immediately restrict to the case when $\Lambda$ is finite dimensional. Then $\Pi_n\cong\tau^{-n}\Lambda$, say as right $\Lambda$-modules, so $\Pi$ is finite dimensional if and only if $\tau^{-n}\Lambda=0$ for $n\gg0$. This is equivalent to saying that all projective modules are postinjective, which happens if and only if $\Lambda$ has finite representation type.
\end{proof}

\subsection{The ext-quiver}

Each preprojective algebra $\Pi$ admits a split algebra homomorphism $\Pi\twoheadrightarrow A$, and so we may again define the Serre subcategory $\nilp\Pi$ of modules generated by the simples $S_i$. Equivalently, writing $\Pi$ as the preprojective algebra of an hereditary algebra $\Lambda$, these are the finite dimensional modules on which $\Lambda_++\Pi_+$ acts nilpotently, or alternatively those on which $\Pi_+$ acts nilpotently and whose retriction lies in $\nilp\Lambda$.

As usual we can construct the associated ext-quiver of $\nilp\Pi$, or more precisely consider the radical-squared zero algebra $\Pi/(\Lambda_++\Pi_+)^2$.

\begin{lem}\label{lem:triv-extn}
We have the trivial extension algebra
\[ \Pi/(\Lambda_++\Pi_+)^2 \cong A\ltimes\big(M\oplus DM\big), \quad \textrm{where }M=\Lambda_+/\Lambda_+^2. \]
In particular, the ext-quiver of $\nilp\Pi$ is the double of the ext-quiver of $\nilp\Lambda$.
\end{lem}

\begin{proof}
We begin by observing that $\Pi/\Pi_+^2\cong\Lambda\ltimes\Pi_1$, and hence that
\[ \Pi/(\Lambda_++\Pi_+)^2 \cong (\Lambda/\Lambda_+^2) \ltimes (A\otimes_\Lambda\Pi_1\otimes_\Lambda A). \]

Assume first that $\Lambda$ is finite dimensional. Then $A\otimes_\Lambda\Pi_1 \cong \tau^{-1}A$, so using that $A\cong DA$ together with the Auslander--Reiten Formula we have
\[ A\otimes_\Lambda\Pi_1\otimes_\Lambda DA \cong D\Hom_\Lambda(\tau^-A,A) \cong \Ext^1_\Lambda(A,A). \]
We also have the isomorphism
\[ \Lambda/\Lambda_+^2 =A\ltimes M \cong A\ltimes D\Ext^1_\Lambda(A,A) \]
shown in \intref{Proposition}{prop:ext-quiver-lambda}. As the ideal $M$ in $A\ltimes M$ acts trivially on $\Ext^1_\Lambda(A,A)$, the claim follows.

Suppose instead that $\Lambda=T_A(M)$ is a tensor algebra, so we can write $\Pi=T_A(M\oplus DM)/\langle c\rangle$. As we are quotienting out by the ideal $\langle M\oplus DM\rangle^2$, and this contains the element $c$, we again get $A\ltimes(M\oplus DM)$.
\end{proof}

As a consequence we observe that
\[ \Ext^1_\Pi(S_i,S_j) \cong \Ext^1_\Lambda(S_i,S_j)\oplus D\Ext^1_\Lambda(S_j,S_i), \]
and hence that
\[ \Ext^1_\Pi(S_j,S_i) \cong D\Ext^1_\Pi(S_i,S_j). \]

\section{Global dimension}

We have now defined the preprojective algebra $\Pi$ for each hereditary algebra $\Lambda$ in one of our classes. This can be constructed either as a tensor algebra $\Pi=T_\Lambda(\Pi_1)$, or equivalently as the quotient of an hereditary tensor algebra $T_\Lambda(\mho)$ by an ideal $I=\langle\bar c\rangle$.

More precisely, starting from $\Lambda=A\oplus\Lambda_+$, we first introduced the $\Lambda$-bimodule $\Omega$ of 1-forms on $\Lambda$ over $A$, defined as the kernel of the multiplication map $\Lambda\otimes_A\Lambda\to\Lambda$. The bimodule $\mho$ is then defined as
\[ \mho \coloneqq D(D\Lambda\otimes_\Lambda\Omega\otimes_\Lambda D\Lambda) \]
whenever $\Lambda$ is finite dimensional, whereas if $\Lambda=T_A(M)$ is a locally finite tensor algebra, then
\[ \Omega = \Lambda\otimes_AM\otimes_A\Lambda \quad\textrm{and}\quad \mho = \Lambda\otimes_ADM\otimes_A\Lambda. \]
In both cases we have that $\Pi_1$ is the cokernel of the bimodule homomorphism
\[ \Lambda\otimes_A\Lambda \to \mho, \quad 1\otimes1 \mapsto \bar c. \]
In fact, for a locally finite hereditary tensor algebra $\Lambda=T_A(M)$ we have $T_\Lambda(\mho)\cong T_A(M\oplus DM)$, leading to a third description
\[ \Pi = T_A(M\oplus DM)/\langle c\rangle, \quad c = \rho-\ell \in (M\otimes_ADM)\oplus(DM\otimes_AM). \]

\subsection{Standard projective presentation}

Let $\Lambda$ be an hereditary algebra in one of our classes, $\Pi$ its preprojective algebra, and $T=T_\Lambda(\mho)$.

\begin{thm}\label{thm:standard-proj-pi}
Every right $\Pi$-module $X$ admits a standard projective presentation
\[ \begin{tikzcd}
X\otimes_A\Pi \arrow[r] & X\otimes_T\widetilde\Omega\otimes_T\Pi \arrow[r] & X\otimes_A\Pi \arrow[r] & X \arrow[r] & 0,
\end{tikzcd} \]
where the map in the left sends $x\otimes1$ to $x\otimes d\bar c\otimes1$.
\end{thm}

\begin{proof}
We write $\Pi=T/I$ where $I=\langle\bar c\rangle$. We begin by tensoring the standard projective resolution of $X_T$, given in \intref{Theorem}{thm:tensor-mho-hered}, with the inclusion $I\rightarrowtail T$, viewed as a map between projective left $T$-modules. We obtain the exact commutative diagram
\[ \begin{tikzcd}
0 \arrow[r] & X\otimes_T\widetilde\Omega\otimes_T I \arrow[r] \arrow[d] & X\otimes_AI \arrow[r] \arrow[d,rightarrowtail]
& X\otimes_TI \arrow[r] \arrow[d,"0"] & 0\\
0 \arrow[r] & X\otimes_T\widetilde\Omega\otimes \arrow[r] & X\otimes_AT \arrow[r] & X \arrow[r] & 0
\end{tikzcd} \]
As $X$ is a $\Pi$-module we have $XI=0$, and hence the vertical map $X\otimes_TI\to X$ on the right is zero. Also, the the vertical map in the middle is injective as $A$ is semisimple. The Snake Lemma now yields the four term exact sequence
\[ 0 \to X\otimes_TI \to X\otimes_T\widetilde\Omega\otimes_T\Pi \to X\otimes_A\Pi \to X \to 0. \]

Next, we know that $I$ is generated as an ideal by the element $\bar c$, given via the bimodule homomorphism
\[ \Lambda\otimes_A\Lambda \to \mho, \quad 1\otimes1 \mapsto \bar c. \]
In particular, as $a\otimes1=1\otimes a$ on the left, we must have that $a\bar c=\bar ca$. Thus there is a natural surjection of right $T$-modules
\[ X\otimes_AT \twoheadrightarrow X\otimes_TI, \quad x\otimes t \mapsto x\otimes\bar ct. \]
Using that $XI=0$ we see that
\[ x\otimes t\bar ct' \mapsto x\otimes\bar ct\bar ct' = x\bar ct\otimes\bar ct' = 0, \]
so the submodule $X\otimes_AI$ is contained in the kernel and we have the induced surjection $X\otimes_A\Pi\twoheadrightarrow X\otimes_TI$. Putting this together yields the desired four term presentation of $X$.

Finally, the map on the left sends the element $x\otimes1$ in $X\otimes_A\Pi$ to $x\otimes\bar c\in X\otimes_TI$. As in the usual diagram chase, we can consider the element $x\otimes\bar c$ in $X\otimes_AT$, which as $x\bar c=0$ is the image of $x\otimes d\bar c$ from $X\otimes_T\widetilde\Omega$. Thus $x\otimes1$ is sent to $x\otimes d\bar c\otimes1$ in $X\otimes_T\widetilde\Omega\otimes_T\Pi$.
\end{proof}

When $\Lambda$ is a tensor algebra we have the following alternative description of the standard projective presentation.

\begin{cor}\label{cor:LHS-tensor}
Let $\Lambda=T_A(M)$ be a locally finite hereditary tensor algebra, $\Pi$ its preprojective algebra, and $X$ a right $\Pi$-module. Then the functorial projective presentation of $X$ becomes
\[ \begin{tikzcd}
X\otimes_A\Pi \arrow[r] & X\otimes_A(M\oplus DM)\otimes_A\Pi \arrow[r] & X\otimes_A\Pi \arrow[r] & X \arrow[r] & 0,
\end{tikzcd} \]
with the left and map sending $x\otimes1$ to the element
\[ \sum (xm_i\otimes\mu_i\otimes1 + x\otimes m_i\otimes\mu_i) - \sum (x\nu_j\otimes n_j\otimes1 + x\otimes\nu_j\otimes n_j). \]
\end{cor}

\begin{proof}
We have $\bar c=\sum m_i\cdot\mu_i - \sum\nu_j\cdot n_j$ in $T=T_A(M\oplus DM)$, so
\[ d\bar c = \sum(m_i\cdot d\mu_i + dm_i\cdot\mu_i) + \sum(\nu_j\cdot dn_j + d\nu_j\cdot n_j). \]
Under the bimodule isomorphism $T\otimes_A(M\oplus DM)\otimes_AT\iso\widetilde\Omega$, sending $1\otimes\gamma\otimes1$ to $d\gamma$ for all $\gamma\in M\oplus DM$, we see that $d\bar c$ corresponds to the element $\rho\otimes1+1\otimes\rho-\ell\otimes1-1\otimes\ell$. The result follows.
\end{proof}

\subsection{Global dimension 2 in non-Dynkin type}

Our aim in this section is to prove the following strengthening of \intref{Theorem}{thm:standard-proj-pi}.

\begin{thm}\label{thm:gl-dim-2}
Let $\Pi$ be an indecomposable preprojective algebra of non-Dynkin type. The every right $\Pi$-module $X$ admits a standard projective resolution
\[ \begin{tikzcd}
0 \arrow[r] & X\otimes_A\Pi \arrow[r] & X\otimes_T\widetilde\Omega\otimes_T\Pi \arrow[r] & X\otimes_A\Pi \arrow[r] & X \arrow[r] & 0
\end{tikzcd} \]
and $\gldim\Pi=2$.
\end{thm}

We observe that it is enough to show the exactness in the special case of $X=\Pi$. For, the sequence from \intref{Theorem}{thm:standard-proj-pi} is functorial, so if the above result holds for $X=\Pi$, then we have an exact sequence with all terms projective on both the left and the right. It therefore remains exact after tensoring on the left with an arbitrary module $X$, which is what we needed to show.

The key to proving this will be the following observation.

\begin{prop}\label{prop:ker-on-left}
The composite
\[ X\otimes_A\Pi \to X\otimes_T\widetilde\Omega\otimes_T\Pi \to X\otimes_\Lambda\mho\otimes_\Lambda\Pi \]
sends $x\otimes1$ to $x\otimes\bar c\otimes1$. If $\Lambda$ is finite dimensional, then viewed as a morphism of right $\Lambda$-modules this map has kernel $\bigoplus_n\Hom_\Lambda(D\Pi_n,X)$.
\end{prop}

\begin{proof}
The map $\widetilde\Omega\to T\otimes_\Lambda\mho\otimes_\Lambda T$ is induced by the composite $\widetilde\Omega\rightarrowtail T\otimes_AT \twoheadrightarrow T\otimes_\Lambda T$. As this sends $dt$ to $1\otimes t-t\otimes1$ for all $t\in T$, we see that $d\gamma\in\widetilde\Omega$ is sent to $1\otimes\gamma\otimes1$ in $T\otimes_\Lambda\mho\otimes_\Lambda T$ for all $\gamma\in\mho$.

It follows that the map $X\otimes_A\Pi\to X\otimes_\Lambda\mho\otimes_\Lambda\Pi$ sends $x\otimes\pi$ to $x\otimes\bar c\otimes\pi$. 
Viewed as a map of right $\Lambda$-modules, this decomposes into the direct sum of the maps
\[ X\otimes_A\Pi_n \to X\otimes_\Lambda\mho\otimes_\Lambda\Pi_n, \quad x\otimes\pi \mapsto x\otimes\bar c\otimes\pi. \]
If $\Lambda$ is finite dimensional, then we can use the left handed version of \intref{Proposition}{prop:AR-translate} to obtain the morphism
\[ \Lambda\otimes_A\Pi_n \to \mho\otimes_\Lambda\Pi_n \cong \Hom_{\,\hyp\Lambda}(D\Pi_n\otimes_\Lambda\Omega,\Lambda) \]
sending $1\otimes\pi$ to the map
\[ \theta\otimes d\lambda \mapsto \lmod{\theta\lambda}(\pi) - \lmod\theta(\pi)\lambda. \]
Moreover, this agrees with the map obtained from applying $\Hom_{\,\hyp\Lambda}(-,\Lambda)$ to the projective resolution of $D\Pi_n$ as a right $\Lambda$-module.

Finally, we know that $D\Pi_n\otimes_\Lambda\Omega$ is a finite dimensional projective right $\Lambda$-module, so \intref{Lemma}{lem:standard-isos} applies and we can tensor on the left with $X$ to get
\[ X\otimes_A\Pi_n \to \Hom_{\,\hyp\Lambda}(D\Pi_n\otimes_\Lambda\Omega,X) \]
sending $x\otimes\pi$ to the map
\[ \theta\otimes d\lambda \mapsto x\lmod{\theta\lambda}(\pi) - x\lmod\theta(\pi)\lambda. \]
This now agrees with the map obtained from applying $\Hom_{\,\hyp\Lambda}(-,X)$ to the projective resolution of $D\Pi_n$. It follows that the kernel is precisely $\Hom_\Lambda(D\Pi_n,X)$.
\end{proof}

\subsubsection{Finite dimensional hereditary algebras}

In this section we consider a finite dimensional, indecomposable hereditary algebra $\Lambda$ of infinite representation type, with preprojective algebra $\Pi$. We will prove \intref{Theorem}{thm:gl-dim-2} in two steps, first showing that we have an exact sequence, so that $\gldim\Pi\leq2$, and then showing that the semisimple module $A$ has projective dimension precisely two.

\begin{prop}\label{prop:gl-dim-2-fd}
We have the following exact sequence
\[ \begin{tikzcd}
0 \arrow[r] & \Pi\otimes_A\Pi \arrow[r] & \Pi\otimes_T\widetilde\Omega\otimes_T\Pi \arrow[r] & \Pi\otimes_A\Pi \arrow[r] & \Pi \arrow[r] & 0.
\end{tikzcd} \]
\end{prop}

\begin{proof}
Using \intref{Theorem}{thm:standard-proj-pi} we only need to show that the left hand map is injective. Now we know from \intref{Proposition}{prop:ker-on-left} that the composite
\[ X\otimes_A\Pi \to X\otimes_T\widetilde\Omega\otimes_T\Pi \to X\otimes_\Lambda\mho\otimes_\Lambda\Pi \]
regarded as a morphism of right $\Lambda$-modules, has kernel $\bigoplus_n\Hom_\Lambda(D\Pi_n,X)$. Next, \intref{Corollary}{cor:dual-AR} tells us that $D(\tau^{-n}\Lambda)\cong\tau^n(D\Lambda)$, so that $D\Pi_n$ is a postinjective $\Lambda$-module. Thus when $X=\Pi$ we can use that there are no non-zero homomorphisms from postinjective modules to preprojective modules to deduce that $\Hom_{\,\hyp\Lambda}(D\Pi_n,\Pi)=0$ for all $n$.

This shows that, when $X=\Pi$, the above composite map is injective, and hence the map on the left is also injective.
\end{proof}

\begin{thm}
The left hand map in the standard projective resolution of $A$, viewed as a morphism of right $\Lambda$-modules, is isomorphic to the direct sum of the source maps for each indecomposable direct summand of $\Pi$. In particular, $\Pi$ has right (and dually left) global dimension two.
\end{thm}

\begin{proof}
Consider the standard projective resolution of $X=A$. This is the Yoneda composite of the short exact sequences
\[ \begin{tikzcd}
0 \arrow[r] & J+\Pi_+ \arrow[r] & \Pi \arrow[r] & A \arrow[r] & 0
\end{tikzcd} \]
and
\[ \begin{tikzcd}
0 \arrow[r] & \Pi \arrow[r] & A\otimes_T\widetilde\Omega\otimes_T\Pi \arrow[r] & J+\Pi_+ \arrow[r] & 0.
\end{tikzcd} \]
We wish to show that the latter sequence is isomorphic, as right $\Lambda$-modules, to the direct sum of the Auslander--Reiten sequences starting at each $\Pi_n=\tau^{-n}\Lambda$ (together with the identity map $J\to J$). As such it will be non-split, so will be non-split as right $\Pi$-modules, and hence $\Pi$ has right global dimension 2. An analogous argument shows that it also has left global dimension 2.

We first recall from \intref{Remark}{rem:omega-not-proj-bimod} that $A\otimes_\Lambda\Omega\cong J$ as right $\Lambda$-modules. Then, using that $A\cong DA$ and \intref{Lemma}{lem:mho-proj}, we have the projective right $\Lambda$-module
\[ A\otimes_\Lambda\mho \cong \Hom_{\Lambda\hyp}(\Omega\otimes_\Lambda A,\Lambda) \cong \Hom_{\Lambda\hyp}(J,\Lambda)
\eqqcolon {}^\vee\! J. \]
Next, as $A\otimes_\Lambda\mho$ is a projective right $\Lambda$-module, the surjection $\mho\twoheadrightarrow A\otimes_\Lambda\mho$ admits a section $\beta$. We can then use this to construct a morphism $\beta\otimes1$ from ${}^\vee\!J\otimes_\Lambda T$ to $T_+\cong A\otimes_\Lambda\widetilde\Omega$ of right $T$-modules, and hence a section of the epimorphism $A\otimes_\Lambda\widetilde\Omega\twoheadrightarrow A\otimes_\Lambda\mho\otimes_\Lambda T$.

After tensoring this map with $\Pi$ we see that we can write the second exact sequence above as
\[ \begin{tikzcd}
0 \arrow[r] & \Pi \arrow[r] & (J\otimes_\Lambda\Pi)\oplus({}^\vee\!J\otimes_\Lambda\Pi) \arrow[r] & J+\Pi_+ \arrow[r] & 0.
\end{tikzcd} \]
As a sequence of right $\Lambda$-modules we see that we can split off the identity map $J\otimes_\Lambda\Lambda\to J$ to obtain
\[ \begin{tikzcd}
0 \arrow[r] & \Pi \arrow[r] & (J\otimes_\Lambda\Pi_+)\oplus({}^\vee\!J\otimes_\Lambda\Pi) \arrow[r] & \Pi_+ \arrow[r] & 0.
\end{tikzcd} \]
It remains to see that the morphisms and compatible with the grading. This is clear for the multiplication map $J\otimes_\Lambda\Pi_+\to\Pi_+$ on the right hand side. For the second component note that we started from the section $\beta\colon{}^\vee\!J\to\mho$, so $\beta\otimes1$ sends ${}^\vee\!J\otimes_\Lambda\Pi_n$ to $\mho\otimes_\Lambda\Pi_n=\Pi_{n+1}$.

Next, $\bar c-\beta(1\otimes\bar c)\in\mho$ is sent to zero in $A\otimes_\Lambda\mho$, so comes from some element $x\in J\otimes_\Lambda\mho$. Thus the left hand map $\Pi\to(J\otimes_\Lambda\Pi_+)\oplus({}^\vee\!J\otimes_\Lambda\Pi)$ sends to the preimage of $\bar c$, so to $x+\beta(1\otimes\bar c\otimes1)$. Thus the left hand map restricts to morphisms $\Pi_n\to(J\otimes_\Lambda\Pi_{n+1})\oplus({}^\vee\!J\otimes_\Lambda\Pi_n)$.

It follows that our sequence decomposes into the direct sum of the short exact sequences
\[ \begin{tikzcd}
0 \arrow[r] & \Pi_n \arrow[r] & (J\otimes_\Lambda\Pi_{n+1})\oplus({}^\vee\!J\otimes_\Lambda\Pi_n) \arrow[r] & \Pi_{n+1} \arrow[r] & 0.
\end{tikzcd} \]
To see that each of these is an Auslander--Reiten sequence we first multiply on the left by the primitive idempotent $e_i$. The resulting short exact sequence lies in $\Ext^1_\Lambda(\tau^{-n-1}P_i,\tau^{-n}P_i)$, which by the Auslander--Reiten Formula is isomorphic to $D\End_\Lambda(P_i)\cong DA_i\cong A_i$. Thus every non-split sequence must be isomorphic to the Auslander--Reiten sequence.

It remains to observe that the sequence is indeed non-split. In fact, there are no non-zero homomorphisms from the middle term to the left hand term. For, using the adjunction, we just need to show that $\Hom_{\,\hyp\Lambda}(\tau^-(e_iJ)\oplus e_i\cdot{}^\vee\!J,P_i)=0$. This is clear for the first component, whereas for the second component we have $e_i\cdot{}^\vee\!J\cong{}^\vee\!(Je_i)$, so
\[ \Hom_{\,\hyp\Lambda}(e_i\cdot{}^\vee\!J,P_i) \cong P_i\otimes_\Lambda\Hom_{\,\hyp\Lambda}({}^\vee\!(Je_i),\Lambda) \cong P_i\otimes Je_i \cong e_iJe_i, \]
which vanishes as $\Lambda$ is finite dimensional.
\end{proof}

\subsubsection{Locally finite hereditary tensor algebras}

We now consider an indecomposable, hereditary, locally finite tensor algebra $\Lambda=T_A(M)$ of infinite representation type, with preprojective algebra $\Pi$. For notational simplicity all unadorned tensor products in this section will be over $A$.

Recall that $\Omega=\Lambda\otimes M\otimes\Lambda$, that $\mho=\Lambda\otimes DM\otimes\Lambda$, and that $\widetilde\Omega=T\otimes_(M\oplus DM)\otimes T$. In particular, the standard projective presentation of $X$ becomes
\[ \begin{tikzcd}
X\otimes_A\Pi \arrow[r] & X\otimes_A(M\oplus DM)\otimes_A\Pi \arrow[r] & X\otimes_A\Pi \arrow[r] & X \arrow[r] & 0.
\end{tikzcd} \]
In the previous section we showed that, when $\Lambda$ is finite dimensional, the second component of the left hand map
\[ \Pi\otimes\Pi \to \Pi\otimes DM\otimes\Pi, \]
viewed as a morphism of right $\Lambda$-modules, decomposes into the direct sum of morphisms $\Pi\otimes\Pi_n\to\Pi\otimes DM\otimes\Pi_n$, each having vanishing kernel $\Hom_{\,\hyp\Lambda}(D\Pi_n,\Pi)=0$.

We therefore need to find an analogous statement when $\Lambda$ is infinite dimensional. For this we will need to incorporate the natural grading on $\Lambda=T_A(M)$. We write $\Mod_\gr\Lambda$ for the category of all graded $\Lambda$-modules, together with degree zero morphisms. Thus, for graded modules $X=\bigoplus X_d$ and $Y=\bigoplus Y_d$, the set $\Hom_\gr(X,Y)$ consists of those $\Lambda$-homomorphisms $f\colon X\to Y$ with $f(X_d)\subseteq Y_d$ for all $d$. We also have the $r$-th twist $Y(r)$ of a graded module $Y$, given by $Y(r)_d=Y_{d+r}$. This then allows us to define
\[ \HOM(X,Y) \coloneqq \bigoplus \Hom_\gr(X,Y(d)) \quad\textrm{and}\quad \EXT(X,Y) \coloneqq \bigoplus \Ext^1_\gr(X,Y(d)). \]

We observe that the category $\Mod_\gr\Lambda$ is again hereditary. In fact, for an $A$-module $W$, we may consider $W\otimes\Lambda$ as a graded projective generated in degree zero, in which case the standard resolution of $X\in\Mod\Lambda$ can be regarded as an exact sequence of graded modules
\[ \begin{tikzcd}
0 \arrow[r] & \bigoplus_d (X_{-d}\otimes M\otimes\Lambda)(d-1) \arrow[r] & \bigoplus_d (X_{-d}\otimes\Lambda)(d) \arrow[r]
& X \arrow[r] & 0.
\end{tikzcd} \]
Alternatively, if we follow the usual convention that $X\otimes\Lambda$ has grading $(X\otimes\Lambda)_n=\bigoplus_{d+e=n}X_d\otimes\Lambda_e$, and similarly $(X\otimes M\otimes\Lambda)_n=\bigoplus_{d+e+1=n}X_d\otimes M\otimes\Lambda_e$ (with the understanding that elements of $M$ have degree one), then we can simply write $0\to X\otimes M\otimes\Lambda\to X\otimes\Lambda\to X\to 0$.

We next observe that $T=T_A(M\oplus DM)$ has a natural bigrading such that $\deg A=(0,0)$, $\deg M=(1,0)$ and $\deg DM=(0,1)$. The element $c=\rho-\ell$ is now homogeneous of degree $(1,1)$, so the preprojective algebra $\Pi=T/\langle c\rangle$ inherits the bigrading. In particular each $\Pi_n=\bigoplus_m\Pi_{(m,n)}$ is a locally finite graded $\Lambda$-bimodule, so we can construct its graded dual
\[ E_n \coloneqq \sideset{}{_m}\bigoplus D\Pi_{(m,n)} \]
which is again a locally finite graded $\Lambda$-bimodule.

Note that, to remain consistent, the defining sequence for $\Pi_1$
\[ \begin{tikzcd}
\Lambda\otimes\Lambda \arrow[r] & \Lambda\otimes DM\otimes\Lambda \arrow[r] & \Pi_1 \arrow[r] & 0
\end{tikzcd} \]
should be a morphism of graded right $\Lambda$-modules, so element of $DM$ have degree $-1$, the element $\bar c$ has degree zero, and $\Pi_1$ has $\Pi_{(m,1)}$ in degree $m-1$. As $\Pi_n=\Pi_1\otimes_\Lambda\Pi_{n-1}$ we see that $\Pi_n$ has $\Pi_{(m,n)}$ in degree $m-n$.
Taking duals then tells is that $E_n$ should have $D\Pi_{(m,n)}$ in degree $n-m$.

We are now set to emulate the proof in the finite dimensional case. We will replace the finite dimensional modules $D\Pi_n=\tau^n(D\Lambda)$ in that case with the graded modules $E_n$, and replace $\Hom_\Lambda(X,Y)$ and $\Ext^1_\Lambda(X,Y)$ by $\HOM(X,Y)$ and $\EXT(X,Y)$.

The first step is to generalise the second part of \intref{Proposition}{prop:ker-on-left}.

\begin{lem}\label{lem:gr-hom-ext}
For a positively graded $\Lambda$-module $X$ we have the exact sequence
\[ \begin{tikzcd}[column sep=17pt]
0 \arrow[r] & \HOM(E_n,X) \arrow[r] & X\otimes\Pi_n \arrow[r] & X\otimes DM\otimes\Pi_n \arrow[r] & \EXT(E_n,X) \arrow[r] & 0
\end{tikzcd} \]
where the morphism in the middle sends $x\otimes\pi$ to $x\rho\otimes\pi - x\otimes\ell\pi$.
\end{lem}

\begin{proof}
The standard projective resolution of $E_n$ is
\[ \begin{tikzcd}
0 \arrow[r] & E_n\otimes M\otimes\Lambda \arrow[r] & E_n\otimes\Lambda \arrow[r] & E_n \arrow[r] & 0
\end{tikzcd} \]
This is an exact sequence of graded $\Lambda$-modules, and since $(E_n)_{-d}=D\Pi_{(n+d,n)}$ we can express the map on the left as
\[ \sideset{}{_d}\bigoplus(D\Pi_{(n+d,n)}\otimes M\otimes\Lambda)(d-1) \to \sideset{}{_d}\bigoplus D\Pi_{(n+d,n)}\otimes\Lambda(d). \]

We now want to apply $\Hom_\gr(-,X(r))$. Given a finite dimensional $A$-module $W$, the usual tensor-hom adjunction shows that
\[ \Hom_\gr(W\otimes\Lambda(d),X(r)) \cong \Hom_A(W,X_{r-d}) \cong X_{r-d}\otimes DW, \]
which vanishes for all $d>r$ as $X$ is positively graded. Thus
\[ \Hom_\gr(E_n\otimes\Lambda,X(r)) \cong \sideset{}{_d}\bigoplus X_{r-d}\otimes\Pi_{(n+d,n)} \cong (X\otimes\Pi_n)_r, \]
where we have used that $(\Pi_n)_d=\Pi_{(n+d,n)}$. Similarly for the term involving $E_n\otimes M\otimes\Lambda$. Taking the sum over all $r$ thus yields the required four term exact sequence.

To compute the morphism we may follow the proof of \intref{Lemma}{lem:bar-c-duality} (2) to see that $x\otimes\pi\in X_{r-d}\otimes\Pi_{(n+d,n)}$ is sent to the pair $(\alpha,-\beta)$, where
\[ \alpha \colon D\Pi_{(n+d,n)}\otimes M \to X_{r+d+1}, \quad \theta\otimes m \mapsto x\lmod\theta(\pi)m \]
and
\[ \beta \colon D\Pi_{(n+d+1,n)}\otimes M \to X_{r+d}, \quad \theta\otimes m \mapsto x\lmod{\theta m}(\pi). \]

Now, as in the proof of \intref{Proposition}{prop:bar-c-tensor}, the element $x\rho=\sum xm_i\otimes\mu_i$ in $X_{r+d+1}\otimes DM$ corresponds to the map $m\mapsto xm$ in $\Hom_A(M,X_{r+d+1})$, and so $x\rho\otimes\pi$ corresponds to the map $\alpha$.

Analogously $\ell\pi\in DM\otimes\Pi_{(n+d,n)}$ yields $m\mapsto m\pi$ in $\Hom_{A\hyp}(M,\Pi_{(n+d+1,n)})$, and hence $\psi\colon\theta\otimes m\mapsto\theta(m\pi)$ in $D(D\Pi_{(n+d+1,n)}\otimes M)$. Now for all $a\in A$ we have
\[ \sigma(\rmod\psi(\theta\otimes m)a) = \psi(\theta\otimes ma) = \theta(ma\pi) = (\theta m)(a\pi) = \sigma(a\lmod{\theta m}(\pi)). \]
Thus $\rmod\psi(\theta\otimes m)=\lmod{\theta m}(\pi)$, and so $x\otimes\ell\pi$ corresponds to the map $\beta$.

We deduce that $x\otimes\pi$ is sent to $x\rho\otimes\pi-x\otimes\ell\pi$.
\end{proof}

The next step is to show that $\HOM(E_n,\Pi_m)=0$ for all $m,n$. The analogous statement in the finite dimensional case followed since there are no non-zero homomorphisms from postinjective modules to preprojective modules.

\begin{lem}\label{lem:zero-HOM}
Assume $\Lambda$ has no finite dimensional right ideals. Then for all $m,n$ we have, in the category of graded right $\Lambda$-modules, that $\HOM(E_n,\Pi_m)=0$ and $\EXT(E_n,\Pi_m)\cong\Pi_{m+n+1}$.
\end{lem}

\begin{proof}
Consider first the case $m=0$, so $\Pi_0=\Lambda$. Since $\Lambda(r)$ is non-zero only in degrees $d\geq-r$, whereas $E_n$ lives in degrees $d\leq n$, we see that the image of any graded right $\Lambda$ homomorphism $E_n\to\Lambda(r)$ is supported in only finitely many degrees, so yields a finite dimensional submodule of $\Lambda$. Our assumption on $\Lambda$ thus implies that $\HOM(E_n,\Lambda)=0$. By \intref{Lemma}{lem:gr-hom-ext} we have the short exact sequence
\[ \begin{tikzcd}
0 \arrow[r] & \Lambda\otimes\Pi_n \arrow[r] & \Lambda\otimes DM\otimes\Pi_n \arrow[r] & \EXT(E_n,\Lambda) \arrow[r] & 0
\end{tikzcd} \]
with the map on the left being $1\otimes\pi\mapsto\rho\otimes\pi-1\otimes\ell\pi$. On the other hand we have the exact sequence from \intref{Proposition}{prop:bar-c-tensor}
\[ \begin{tikzcd}
\Lambda\otimes\Lambda \arrow[r] & \Lambda\otimes DM\otimes\Lambda \arrow[r] & \Pi_{(1)} \arrow[r] & 0,
\end{tikzcd} \]
where the left hand map sends $1\otimes1$ to $\rho\otimes1-1\otimes\ell$. Comparing these in the case $n=0$ reveals that $\Pi_1\cong\EXT(E_0,\Lambda)$ and that the latter sequence is exact. We may now tensor the latter sequence on the right with $\Pi_n$, recalling that $\Pi_a\otimes_\Lambda\Pi_b\cong\Pi_{a+b}$, to get
\[ \begin{tikzcd}[column sep=17pt]
0 \arrow[r] & \Tor_1^\Lambda(\Pi_1,\Pi_n) \arrow[r] & \Lambda\otimes\Pi_n \arrow[r] & \Lambda\otimes DM\otimes\Pi_n \arrow[r]
& \Pi_{n+1} \arrow[r] & 0.
\end{tikzcd} \]
Again, comparing this to the first sequence above shows that $\Pi_{n+1}\cong\EXT(E_n,\Lambda)$ and $\Tor_1^\Lambda(\Pi_1,\Pi_n)=0$, so we have the short exact sequence
\[  \begin{tikzcd}
0 \arrow[r] & \Lambda\otimes\Pi_n \arrow[r] & \Lambda\otimes DM\otimes\Pi_n \arrow[r] & \Pi_{n+1} \arrow[r] & 0.
\end{tikzcd} \]
If we tensor this on the right with $\Pi_m$, then the same reasoning shows that $\Tor_1^\Lambda(\Pi_n,\Pi_m)=0$ for all $m,n\geq1$. If instead we tensor on the left with $\Pi_m$, then we can use the vanishing of tor to obtain the short exact sequence
\[ \begin{tikzcd}
0 \arrow[r] & \Pi_m\otimes\Pi_n \arrow[r] & \Pi_m\otimes DM\otimes\Pi_n \arrow[r] & \Pi_{m+n+1} \arrow[r] & 0.
\end{tikzcd} \]
Again, the left hand map sends $\pi\otimes\pi'$ to $\pi\rho\otimes\pi'-\pi\otimes\ell\pi'$, so we can invoke \intref{Lemma}{lem:gr-hom-ext} once more to conclude that $\HOM(E_n,\Pi_m)=0$ and $\EXT(E_n,\Pi_m)\cong\Pi_{m+n+1}$.
\end{proof}

With these preparations, we can now prove \intref{Theorem}{thm:gl-dim-2} when $\Lambda=T_A(M)$ is a locally finite hereditary tensor algebra. Again, we split the proof into two parts.

\begin{prop}\label{prop:gl-dim-2-tensor}
We have the following exact sequence
\[ \begin{tikzcd}
0 \arrow[r] & \Pi\otimes_A\Pi \arrow[r] & \Pi\otimes(M\oplus DM)\otimes\Pi \arrow[r] & \Pi\otimes_A\Pi \arrow[r] & \Pi \arrow[r] & 0.
\end{tikzcd} \]
\end{prop}

\begin{proof}
We recall from \intref{Lemma}{lem:indep-orient} that the preprojective algebra is independent of the orientation of $\Lambda$. Suppose first that $e_iMe_i=0$ for all $i$ (no vertex loops). Then we can choose an orientation such that $\Lambda$ is finite dimensional, in which case we know the result. In the remaining cases we fix a vertex $i$ having $e_iMe_i\neq0$, and choose an orientation such that there is a path from every other vertex to $i$. It follows that every non-zero projective right module is infinite dimensional.

Again, we only need to show that the left hand map is injective, and by \intref{Corollary}{cor:LHS-tensor} this has second component sending $\pi\otimes\pi'$ to $\pi\rho\otimes\pi'-\pi\otimes\ell\pi'$ in $\Pi\otimes DM\otimes\Pi$. As a morphism of graded right $\Lambda$-modules we may therefore decompose this into the direct sum of maps $\Pi_m\otimes\Pi_n\to\Pi_m\otimes DM\otimes\Pi_n$. The latter map has kernel $\HOM(E_m,\Pi_n)$ by \intref{Lemma}{lem:gr-hom-ext}, which this vanishes by \intref{Lemma}{lem:zero-HOM}. The result follows.
\end{proof}

\begin{lem}
The left hand map $\Pi\to(M\oplus DM)\otimes\Pi$ in the projective resolution of the semisimple module $A$ is not split. In particular, $\gldim\Pi=2$.
\end{lem}

\begin{proof}
By \intref{Corollary}{cor:LHS-tensor} the left hand map sends $1$ to $c=\sum m_i\otimes\mu_i-\sum\nu_j\otimes n_j$. Now any morphism $\alpha$ in the reverse direction comes from an $A$-linear map $M\oplus DM\to\Pi$, so $\alpha(c)$ lies in $\Pi(M\oplus DM)$, and hence $\alpha$ cannot be a retract.
\end{proof}

\subsection{Euler form}

Let $\Pi$ be a preprojective algebra of non-Dynkin type, say coming from an hereditary algebra $\Lambda$. We have seen that $\Pi$ has global dimension two, so we can compute the Euler form of any two finite dimensional modules $X$ and $Y$ as
\[ \dim\Hom_\Pi(X,Y) - \dim\Ext^1_\Pi(X,Y) + \dim\Ext^2_\Pi(X,Y). \]

We next recall from the discussion in \intref{Section}{sec:ext-quiver-hered} that restriction of scalars yields an additive map $K_0(\fd\Lambda)\to K_0(A)$, inducing an isomorphism $K_0(\nilp\Lambda)\iso K_0(A)$. Moreover, the Euler form on $\fd\Lambda$ descends to a bilinear form $\langle-,-\rangle$ on $K_0(A)$, represented by the matrix $\mat D-\mat R$. This therefore depends only on the ext-quiver of $\nilp\Lambda$, or equivalently on the trivial extension algebra $A\ltimes M$, where $M=\Lambda_+/\Lambda_+^2$.

Similarly, restriction of scalars yields an additive map $K_0(\fd\Pi)\to K_0(A)$, inducing an isomorphism $K_0(\nilp\Pi)\iso K_0(A)$. Also, the ext-quiver of $\nilp\Pi$ is determined by the trivial extension algebra $A\ltimes(M\oplus DM)$, as shown in \intref{Lemma}{lem:triv-extn}. This in turn determines a symmetric bilinear form $(-,-)$ on $K_0(\modcat A)$, which is precisely the symmetrised version of the bilinear form $\langle-,-\rangle$, so is represented by the matrix $2\mat D-\mat B$.

\begin{prop}
Let $\Pi$ be a preprojective algebra of non-Dynkin type. Then the Euler form of two finite dimensional modules depends only on their classes in $K_0(A)$, and is given by the symmetrised bilinear form $(-,-)$.
\end{prop}

\begin{proof}
For finite dimensional modules $X$ and $Y$, the extension groups $\Ext^i_\Pi(X,Y)$ can be computed as the cohomology of the complex
\[ \begin{tikzcd}
\Hom_A(X,Y) \arrow[r] & \Hom_T(X\otimes_T\widetilde\Omega,Y) \arrow[r] & \Hom_A(X,Y)
\end{tikzcd} \]
given by applying $\Hom_\Pi(-,Y)$ to the standard projective resolution of $X$. Thus the Euler form, which is the alternating sum of the dimensions of the extension groups, equals
\[ 2\dim\Hom_A(X,Y) - \dim\Hom_T(X\otimes_T\widetilde\Omega,Y). \]
Next, we have a split exact sequence
\[ \begin{tikzcd}
 0 \arrow[r] & X\otimes_\Lambda\Omega\otimes_\Lambda T \arrow[r] & X\otimes_T\widetilde\Omega \arrow[r] &
X\otimes_\Lambda\mho\otimes_\Lambda T \arrow[r] & 0
\end{tikzcd} \]
so the second dimension equals
\[ \dim\Hom_\Lambda(X\otimes_\Lambda\Omega,Y)+\dim\Hom_\Lambda(X\otimes_\Lambda\mho,Y). \]
Finally, we have the isomorphisms
\[ \Hom_\Lambda(X\otimes_\Lambda\mho,Y) \cong D(X\otimes_\Lambda\mho\otimes_\Lambda DY) \cong
D\Hom_\Lambda(Y\otimes_\Lambda\Omega,X), \]
so the result follows from \intref{Proposition}{prop:ext-quiver-lambda}, where we computed that
\[ \langle X,Y\rangle = \dim\Hom_A(X,Y) - \dim\Hom_\Lambda(X\otimes_\Lambda\Omega,Y) \]
depends only the classes of $X$ and $Y$ in $K_0(A)$.
\end{proof}

Our next result is an important strengthening of our earlier \intref{Proposition}{prop:ker-on-left}.

\begin{thm}
Let $\Pi$ be a preprojective algebra of an hereditary tensor algebra of non-Dynkin type. Then there is a natural isomorphism of bifunctors
\[ \Ext^2_\Pi(X,Y) \cong D\Hom_\Pi(Y,X) \]
for all finite dimensional modules $X$ and $Y$.
\end{thm}

As an immediate corollary we see that the Euler form satisfies
\[ (X,Y) = \dim\Hom_\Pi(X,Y)+\dim\Hom_\Pi(Y,X)-\dim\Ext^1_\Pi(X,Y). \]
As the Euler form is symmetric we deduce the following result.

\begin{cor}
Let $\Pi$ be a preprojective algebra on non-Dynkin type. Then for all finite dimensional modules $X$ and $Y$ we have
\[ \dim\Ext^1_\Pi(X,Y) = \dim\Ext^1_\Pi(Y,X). \]
\end{cor}

We expect the result to hold for all preprojective algebras of non-Dynkin type, but we can only prove this under the assumption that the perfect pairing describing the Auslander--Reiten pairing is invariant under the inverse translate.

\subsubsection{Locally finite hereditary tensor algebras}

\begin{prop}\label{prop:ker-on-left-2}
Let $\Pi$ be the preprojective algebra of $\Lambda=T_A(M)$, and $X$ and $Y$ two right $\Pi$-modules with $Y$ finite dimensional. Tensoring the standard projective presentation of $X$ with $DY$, the left hand map
\[ X\otimes_ADY \to X\otimes_A(M\oplus DM)\otimes_ADY \]
has kernel $\Hom_\Pi(Y,X)$.
\end{prop}

\begin{proof}
The morphism in question is given by
\[ x\otimes\eta \mapsto \sum(xm_i\otimes\mu_i\otimes\eta+x\otimes m_i\otimes\mu_i\eta) -
\sum(x\nu_j\otimes n_j\otimes\eta+x\otimes\nu_j\otimes n_j\eta). \]

By \intref{Lemma}{lem:bar-c-duality} (2) the morphism
\[ \Hom_A(Y,\Lambda) \to \Hom_\Lambda(Y\otimes_\Lambda\Omega,\Lambda) \]
sending $f$ to the map $y\otimes d\lambda\mapsto f(y)\lambda-f(y\lambda)$ corresponds, under the usual identifications, to
\[ \Lambda\otimes_ADY \to \mho\otimes_\Lambda DY, \quad 1\otimes\eta \mapsto \bar c\otimes\eta. \]
If we now tensor with $X$ and use \intref{Lemma}{lem:standard-isos} (1), we see that the map
\[ X\otimes_ADY \to X\otimes_\Lambda\mho\otimes_\Lambda DY, \quad x\otimes\eta \mapsto x\otimes\bar c\otimes\eta, \]
corresponds to the morphism
\[ \Hom_A(Y,X) \to \Hom_\Lambda(Y\otimes_\Lambda\Omega,X) \]
sending $f$ to $y\otimes d\lambda\mapsto f(y)\lambda-f(y\lambda)$.

Using the identifications $\Omega=\Lambda\otimes_AM\otimes_A\Lambda$ and $\mho=\Lambda\otimes_ADM\otimes_A\Lambda$, so that $\bar c=\rho\otimes1-1\otimes\ell$ as shown in \intref{Proposition}{prop:bar-c-tensor}, we can write this as a map $X\otimes_ADY \to X\otimes_ADM\otimes_ADY$, where
\[ x\otimes\eta \mapsto \sum xm_i\otimes\mu_i\otimes\eta - \sum x\otimes\nu_j\otimes n_j\eta. \]
In particular, this morphism has kernel $\Hom_\Lambda(X,Y)$, so those $A$-linear maps $X\to Y$ commuting with multiplication by elements of $M$.

We now consider the hereditary algebra $\tilde\Lambda=T_A(DM)$. In this case the bimodule of noncommutative 1-forms is $\tilde\Lambda\otimes_ADM\otimes_A\tilde\Lambda$, and so the Auslander--Reiten translate is the quotient of the bimodule $\tilde\Lambda\otimes_AM\otimes_A\tilde\Lambda$ by the sub-bimodule generated by
\[ \sum\nu_j\otimes n_j\otimes1-\sum1\otimes m_i\otimes\mu_i. \]
The analogous proof thus shows that the morphism $X\otimes_ADY\to X\otimes_AM\otimes_ADY$ given by
\[ x\otimes\eta \mapsto \sum x\nu_j\otimes n_j\otimes\eta - \sum x\otimes m_i\otimes\mu_i\eta \]
has kernel $\Hom_{\tilde\Lambda}(X,Y)$, so those $A$-linear maps $X\to Y$ commuting with multiplication by elements of $DM$.

Putting this together we see that the kernel of the map $X\otimes_ADY\to X\otimes_A(M\oplus DM)\otimes_ADY$ we are interested in consists of those $A$-linear maps $X\to Y$ which commute with multiplication by all elements of $M\oplus DM$, which as $T_A(M\oplus DM)\twoheadrightarrow\Pi$ is precisely the set $\Hom_\Pi(X,Y)$.
\end{proof}

\begin{cor}
Let $\Pi$ be the preprojective algebra of $\Lambda=T_A(M)$, not of Dynkin type. Then there is a natural isomorphism of bifunctors
\[ \Ext^2_\Pi(X,Y) \cong D\Hom_\Pi(Y,X) \]
for all finite dimensional modules $X$ and $Y$.
\end{cor}

\begin{proof}
Applying $\Hom_\Pi(-,Y)$ to the projective presentation of $X$ yields the morphism
\[ \Hom_A(X\otimes_A(M\oplus DM),Y) \to \Hom_A(X,Y) \]
sending $f$ to the morphism
\[ \hat f \colon x \mapsto \sum f(xm_i\otimes\mu_i) + \sum f(x\otimes m_i)\mu_i - \sum f(x\nu_j\otimes n_j) - \sum f(x\otimes\nu_j)n_j. \]
Clearly this has cokernel $\Ext^2_\Pi(X,Y)$.

We next have the natural isomorphism $\Hom_A(X,Y) \cong D(X\otimes_ADY)$ sending a map $g\colon X\to Y$ to the linear functional $\psi(g)\colon x\otimes\eta\mapsto\eta(g(x))$. Doing the same for the first term, we see that for $f\in\Hom_A(X\otimes_A(M\oplus DM),Y)$ we have
\begin{align*}
\psi(\hat f)(x\otimes\eta) &= \sum \eta(f(xm_i\otimes\mu_i)) + \sum (\mu_i\eta)(f(x\otimes m_i))\\
&\qquad\qquad\qquad - \sum \mu(f(x\nu_j\otimes n_j)) - \sum (n_j\eta)(f(x\otimes\nu_j))\\
&= \sum \psi(f)(xm_i\otimes\mu_i\otimes\eta) + \sum \psi(f)(x\otimes m_i\otimes\mu_i\eta)\\
&\qquad\qquad\qquad - \sum \psi(f)(x\nu_j\otimes n_j\otimes\eta)
- \sum \psi(f)(x\otimes\nu_j\otimes n_j\eta).
\end{align*}
The corresponding morphism $D(X\otimes_A(M\oplus DM)\otimes_ADY) \to D(X\otimes_ADY)$ is thus dual to the morphism considered in the previous proposition, and hence has cokernel $D\Hom_\Pi(Y,X)$.

As all the constructions are natural isomorphisms of bifunctors, so too is the induced isomorphism $\Ext^2_\Pi(X,Y)\cong D\Hom_\Pi(Y,X)$.\end{proof}

\subsubsection{Finite dimensional hereditary algebras}

We now turn our attention to when $\Pi$ is the preprojective algebra of a finite dimensional hereditary algebra $\Lambda$. In this case we need a different approach, as it is not clear how to adapt the proof of \intref{Proposition}{prop:ker-on-left-2}, or what amounts to the same thing, how to produce a perfect pairing between $X\otimes_T\widetilde\Omega\otimes_TDY$ and $Y\otimes_T\widetilde\Omega\otimes_TDX$.

Recall that the Auslander--Reiten pairing is a bifunctorial perfect pairing
\[ \{-,-\} \colon \Ext^1_\Lambda(X,Y) \times \Hom_\Lambda(Y,\tau X) \to k. \]
We need to assume that this pairing is invariant under $\tau^-$, so for all $X$ having no non-zero injective direct summands
\[ \{\tau^-\eta,u_Xh\} = \{\eta,(\tau h)u_Y\} \quad\textrm{for } \eta\in\Ext^1_\Lambda(X,Y) \quad\textrm{and}\quad h\in\Hom_\Lambda(\tau^-Y,X), \]
where $u_X\colon X\to\tau\tau^-X$ is the unit of the adjunction. We note that the assumption on $X$ implies that $\tau^-\eta\in\Ext^1_\Lambda(\tau^-X,\tau^-Y)$ for all $\eta\in\Ext^1_\Lambda(X,Y)$.

Under this assumption we can prove the result. Our proof, analogously to say the proof of Serre duality for a smooth curve, first shows the result when one of the modules lies in a certain subcategory (for curves this would be the subcategory of locally free sheaves), and then deduces the result in general by constructing suitable approximations (every torsion sheaf has a presentation using locally free sheaves).

The subcategory we will use is $\fd_{\mathcal I}\Pi$ consisting of those finite dimensional modules whose restriction to $\Lambda$ contain no non-zero postinjective direct summands. We note that this subcategory is closed under kernels and extensions.

\begin{prop}
Let $\Pi$ be the preprojective algebra of a finite dimensional hereditary algebra $\Lambda$, not of Dynkin type. Under the assumption that the Auslander--Reiten pairing is invariant under $\tau^-$ there are natural isomorphisms of bifunctors
\[ \Ext^2_\Pi(X,Y) \cong D\Hom_\Pi(Y,X) \quad\textrm{and}\quad \Ext^2_\Pi(Y,X) \cong D\Hom_\Pi(X,Y). \]
for all modules $X\in\fd_{\mathcal I}\Pi$ and $Y\in\fd\Pi$.
\end{prop}

\begin{proof}
By assumption on $X$ every short exact sequence of $\Lambda$-modules ending in $X$ remains exact after applying $\tau^{-n}$, equivalently after tensoring with $\Pi_n$, so $\Tor_1^\Lambda(\Pi,X)=0$. In particular, starting from the standard projective presentation of $X_\Lambda$
\[ \begin{tikzcd}
\xi\colon &[-20pt] 0 \arrow[r] & X\otimes_\Lambda\Omega \arrow[r] & X\otimes_A\Lambda \arrow[r] & X \arrow[r] & 0
\end{tikzcd} \]
we obtain the short exact sequence of $\Pi$-modules
\[ \begin{tikzcd}
\xi\otimes1\colon &[-20pt] 0 \arrow[r] & X\otimes_\Lambda\Omega\otimes_\Lambda\Pi \arrow[r] & X\otimes_A\Pi \arrow[r] & X\otimes_\Lambda\Pi \arrow[r] & 0.
\end{tikzcd} \]
We deduce that $X\otimes_\Lambda\Pi$ has projective dimension at most one, and after applying $\Hom_\Pi(-,Y)$ we obtain the isomorphism
\[ \Ext^1_\Pi(X\otimes_\Lambda\Pi,Y) \cong \Ext^1_\Lambda(X,Y). \]
We note that these constructions are all functorial in both $X$ and $Y$.

We now compute the inverse $\Theta$ of the final isomorphism explicitly. Given an extension $\eta\in\Ext^1_\Lambda(X,Y)$, we can represent it as the pushout $h\xi$ along some $h\colon X\otimes_\Lambda\Omega\to Y$. This in turn induces a $\Pi$-linear homomorphism $X\otimes_\Lambda\Omega\otimes_\Lambda\Pi\to Y$, and taking the pushout yields an extension in $\Ext^1_\Pi(X\otimes_\Lambda\Pi,Y)$. We can write the morphism as $g(h\otimes1)$, where $g\colon Y\otimes_\Lambda\Pi\to Y$ is the usual action, and clearly $(h\otimes1)(\xi\otimes1)=\eta\otimes1$. Thus $\Theta(\eta)=g(\eta\otimes1)$ is the pushout along the action map $g$ of the exact sequence $\eta\otimes1\in\Ext^1_\Pi(X\otimes_\Lambda\Pi,Y\otimes_\Lambda\Pi)$.

Next, as $\Pi=T_\Lambda(\Pi_1)$ is a tensor algebra, we have the standard exact sequence
\[ \begin{tikzcd}
0 \arrow[r] & X\otimes_\Lambda\Pi_1\otimes_\Lambda\Pi \arrow[r,"\delta_X"] & X\otimes_\Lambda\Pi \arrow[r] & X \arrow[r] & 0.
\end{tikzcd} \] 
Applying $\Hom_\Pi(-,Y)$ and using what we have just shown yields the exact sequence
\[ \begin{tikzcd}
\Ext^1_\Lambda(X,Y) \arrow[r] & \Ext^1_\Lambda(X\otimes_\Lambda\Pi_1,Y) \arrow[r] & \Ext^2_\Pi(X,Y) \arrow[r] & 0.
\end{tikzcd} \]
The map $\delta_X$ sends $x\otimes\pi\otimes1$ to $x\otimes\pi-x\pi\otimes1$, so writing $f\colon X\otimes_\Lambda\Pi\to X$ for the usual action yields the induced map
\[ \delta_X^\ast \colon \Hom_\Lambda(X,Y) \to \Hom_\Lambda(X\otimes_\Lambda\Pi_1,Y), \quad \alpha \mapsto g(\alpha\otimes1)-\alpha f. \]
We claim that we analogously have
\[ \delta_X^\ast \colon \Ext^1_\Lambda(X,Y) \to \Ext^1_\Lambda(X\otimes_\Lambda\Pi_1,Y), \quad \eta \mapsto g(\eta\otimes1)-\eta f. \]

The second component of $\delta_X^\ast$ is the pullback along $f\otimes1\colon X\otimes_\Lambda\Pi_1\otimes_\Lambda\Pi\to X\otimes_\Lambda\Pi$, and clearly
\[ \Theta(\eta)(f\otimes1) = g(\eta\otimes1)(f\otimes1) = g(\eta f\otimes1) = \Theta(\eta f), \]
so corresponds simply to the pullback along $f$.

The first component of $\delta_X^\ast$ is the pullback along $1\otimes m$, where $m\colon\Pi_1\otimes_\Lambda\Pi\rightarrowtail\Pi$ is the multiplication. Now, given a short exact sequence
\[ \begin{tikzcd}
\eta\colon &[-20pt] 0 \arrow[r] & Y \arrow[r] & E \arrow[r] & X \arrow[r] & 0
\end{tikzcd} \]
we have the exact commutative diagram
\[ \begin{tikzcd}
0 \arrow[r] & Y\otimes_\Lambda\Pi_1\otimes_\Lambda\Pi \arrow[r] \arrow[d,"1\otimes m"] &
E\otimes_\Lambda\Pi_1\otimes_\Lambda\Pi \arrow[r] \arrow[d,"1\otimes m"] &
X\otimes_\Lambda\Pi_1\otimes_\Lambda\Pi \arrow[r] \arrow[d,"1\otimes m"] & 0\\
0 \arrow[r] & Y\otimes_\Lambda\Pi \arrow[r] & E\otimes_\Lambda\Pi \arrow[r] &
X\otimes_\Lambda\Pi \arrow[r] & 0
\end{tikzcd} \]
It follows that the pushout of the top row equals the pullback of the bottom row, so $(1\otimes m)(\eta\otimes1\otimes1)=(\eta\otimes1)(1\otimes m)$. Next, we have $g(1\otimes m)=g(g\otimes1)$ as maps $Y\otimes_\Lambda\Pi_1\otimes_\Lambda\Pi\to Y$, so
\[ \Theta(\eta)(1\otimes m) = g(\eta\otimes1)(1\otimes m) = g(1\otimes m)(\eta\otimes1\otimes1) = g(g(\eta\otimes1)\otimes1) = \Theta(g(\eta\otimes1)). \]
This proves the claim.

Using the isomorphism $X\otimes_\Lambda\Pi_1\cong\tau^-X$ we see that $\tau^-h$ corresponds to $h\otimes1$ for all $h\in\Hom_\Lambda(X,Y)$, and similarly for $\eta\in\Ext^1_\Lambda(X,Y)$. Thus we can rephrase the above results as saying
\[ \delta_X^\ast \colon \Hom_\Lambda(X,Y) \to \Hom_\Lambda(\tau^-X,Y), \quad h \mapsto g(\tau^-h) - hf \]
and
\[ \delta_X^\ast \colon \Ext^1_\Lambda(X,Y) \to \Ext^1_\Lambda(\tau^-X,Y), \quad \eta \mapsto g(\tau^-\eta) - \eta f. \]

Finally, we make some computations using the Auslander--Reiten pairing. We have the perfect pairing
\[ \{-,-\} \colon \Ext^1_\Lambda(X,Y)\times\Hom_\Lambda(Y,\tau X) \to k, \]
which together with the adjunction $\Hom_\Lambda(\tau^-Y,X)\cong\Hom_\Lambda(Y,\tau X)$, $h\mapsto(\tau h)u_Y$, yields a perfect pairing
\[ \{-,-\}' \colon \Ext^1_\Lambda(X,Y)\times\Hom_\Lambda(\tau^-Y,X) \to k. \]
Now, taking $\eta\in\Ext^1_\Lambda(Y,X)$ and $h\in\Hom_\Lambda(\tau^-Y,X)$ we have
\begin{align*}
\{ \eta,\delta^\ast_Y(h) \}' &= \{ \eta,f(\tau^-h) \} - \{ \eta,hg \}'\\
&= \{ \eta,(\tau f)(\tau\tau^-h)u_Y \} - \{ \eta,\tau(hg)u_Y \} \\
&= \{ \eta,(\tau f)u_Xh \} - \{ \tau^-\eta,u_Xhg \}\\
&= \{ \eta f,u_Xh \} - \{ g(\tau^-\eta),u_Xh \} = -\{\delta^\ast_X(\eta),u_Xh\} 
\end{align*}
Here the second row uses the adjunction, the third row uses the identity $(\tau\tau^-h)u_Y=u_Xh$ together with the assumed invariance under $\tau^-$, and the fourth row follows from the bifunctoriality of the perfect pairing. We deduce that we have an exact commutative diagram
\[ \begin{tikzcd}
\Ext^1_\Lambda(X,Y) \arrow[r,"\delta^\ast_X"] \arrow[d,"\wr"] & \Ext^1_\Lambda(\tau^-X,Y) \arrow[r] \arrow[d,"\wr"] &
\Ext^2_\Pi(X,Y) \arrow[r] \arrow[d,"\wr","\exists" swap] & 0\\
D\Hom_\Lambda(\tau^-Y,X) \arrow[r,"D\delta^\ast_Y"] & D\Hom_\Lambda(Y,X) \arrow[r] & D\Hom_\Pi(Y,X) \arrow[r] & 0
\end{tikzcd} \]
where the vertical isomorphisms on the left and in the middle are induced by the Auslander--Reiten formula. It follows that we have an induced natural isomorphism of bifunctors between the cokernels
\[ \Ext^2_\Pi(X,Y) \iso D\Hom_\Pi(Y,X) \]
as required.

Similarly, applying $\Hom_\Pi(-,Y)$ to the standard exact sequence for $X$ given above yields the exact sequence
\[ \begin{tikzcd}
0 \arrow[r] & \Hom_\Pi(X,Y) \arrow[r] & \Hom_\Lambda(X,Y) \arrow[r,"\delta_X^\ast"] & \Hom_\Lambda(\tau^-X,Y)
\end{tikzcd} \]
whereas applying $\Hom_\Pi(-,X)$ to the analogous sequence for $Y$ and using that $Y\otimes_\Lambda\Pi$ has projective dimension at most one yields the exact sequence
\[ \begin{tikzcd}
\Ext^1_\Lambda(Y,X) \arrow[r,"\delta_Y^\ast"] & \Ext^1_\Lambda(\tau^-Y,X) \arrow[r] & \Ext^2_\Pi(Y,X) \arrow[r] & 0
\end{tikzcd} \]
As we just shown, the maps $\delta^\ast_X$ and $\delta^\ast_Y$ are adjoint to one another using the Auslander--Reiten pairings, so we can deduce the natural isomorphism of bifunctors
\[ \Ext^2_\Pi(Y,X) \iso D\Hom_\Pi(X,Y). \qedhere \] 
\end{proof}

To get the full result we show that every finite dimensional module admits an approximation by a module in $\fd_{\mathcal I}\Pi$.

\begin{lem}
Let $\Pi$ be the preprojective algebra of a finite dimensional hereditary algebra $\Lambda$, not of Dynkin type. Then every finite dimensional $\Pi$-module $X$ admits a resolution $0\to Z\to Y\to X\to 0$ with $Y,Z\in\fd_{\mathcal I}\Pi$.
\end{lem}

\begin{proof}
Recall that $\Pi=T_\Lambda(\Pi_1)$ is a tensor algebra, and $X\otimes_\Lambda\Pi_1\cong \tau^-X$, so right $\Pi$-modules can be regarded as pairs $(X,f)$ consisting of a right $\Lambda$-module $X$ and a $\Lambda$-homomorphism $f\colon\tau^-X\to X$. Morphisms $(Y,g)\to(X,f)$ of such pairs are given by those $\Lambda$-homomorphisms $p\colon Y\to X$ for which $f(\tau^-p)=pg$, and a sequence of $\Pi$-modules is exact if and only if its restriction is an exact sequence of $\Lambda$-modules.

Now let $(X,f)$ be any finite dimensional $\Pi$-module. We can write $X=X'\oplus I_n\oplus\cdots\oplus I_0$ for some $X'\in\fd_{\mathcal I}\Pi$ and $I_i\in\add\tau^i(D\Lambda)$. The morphism $f\colon\tau^-X\to X$ cannot include any non-zero component from $\tau^-I_r$ to either $X'$ or to $I_s$ for $r\geq s$, and so we have surjections of $\Pi$-modules
\[ (X,f)=(X_{-1},f_{-1}) \twoheadrightarrow (X_0,f_0) \twoheadrightarrow\cdots\twoheadrightarrow (X_n,f_n)=(X',f'), \]
where $X_i=X'\oplus I_n\oplus\cdots\oplus I_{i+1}$ and $f_i$ is the induced morphism.

Our proof will be by descending induction on $i$, with the case $i=n$ being trivial. Suppose therefore that we have an epimorphism $q\colon(Y,g)\twoheadrightarrow(X_r,f_r)$, so that $f_r(\tau^-q)=qg$, and consider the module $(X_{r-1},f_{r-1})$. We have $X_{r-1}=X_r\oplus I_r$, and the morphism $f_{r-1}$ has non-zero components $f_r$ and $h\colon\tau^-X_r\to I_r$.

Taking an epimorphism $p\colon P\twoheadrightarrow I_r$ with $P$ projective we obtain the commutative square
\[ \begin{tikzcd}
\tau^-Y\oplus\tau^{-2}Y\oplus\tau^-P \arrow[r,"\tau^-\pi"] \arrow[d,"\tilde g"] & \tau^-X_r\oplus\tau^-I_r \arrow[d,"f_{r-1}"]\\
Y\oplus\tau^-Y\oplus P \arrow[r,"\pi"] & X_r\oplus I_r
\end{tikzcd} \]
where
\[ \pi = \begin{bmatrix}q&0&0\\0&h\tau^-q&p\end{bmatrix}, \quad f_{r-1} = \begin{bmatrix}f_r&0\\h&0\end{bmatrix}, \quad
\tilde g = \begin{bmatrix}g&0&0\\1&0&0\\0&0&0\end{bmatrix}. \]
The left hand side thus determines a module $(\tilde Y,\tilde g)$ in $\fd_{\mathcal I}\Pi$, and $\pi\colon(\tilde Y,\tilde g)\twoheadrightarrow(X_{r-1},f_{r-1})$ is an epimorphism.

By induction we obtain an epimorphism $(Y',g')\twoheadrightarrow(X,f)$ with $(Y',g')\in\fd_{\mathcal I}\Pi$. As $\fd_{\mathcal I}\Pi$ is closed under submodules, the result follows..
\end{proof}

\begin{cor}
Let $\Pi$ be the preprojective algebra of a finite dimensional hereditary algebra, not of Dynkin type. Under the assumption that the Auslander--Reiten pairing is invariant under $\tau^-$ there is a natural isomorphism of bifunctors
\[ \Ext^1_\Pi(X,Y) \cong D\Hom_\Pi(Y,X) \]
for all finite dimensional modules $X$ and $Y$.
\end{cor}

\begin{proof}
By the lemma there is a short exact sequence
\[ \begin{tikzcd}
0 \arrow[r] & W \arrow[r] & Z \arrow[r] & Y \arrow[r] & 0
\end{tikzcd} \]
with $W,Z\in\fd_{\mathcal I}\Pi$. Applying both $\Hom_\Pi(X,-)$ and $\Hom_\Pi(-,X)$ and using the theorem, we obtain an exact commutative diagram
\[ \begin{tikzcd}
\Ext^2_\Pi(X,W) \arrow[r] \arrow[d,"\wr"] & \Ext^2_\Pi(X,Z) \arrow[r] \arrow[d,"\wr"] & \Ext^2_\Pi(X,Y) \arrow[r] & 0\\
D\Hom_\Pi(W,X) \arrow[r] & D\Hom_\Pi(Z,X) \arrow[r] & D\Hom_\Pi(Y,X) \arrow[r] & 0
\end{tikzcd} \]
Note that the left hand square commutes as we have natural isomorphisms of bifunctors. We therefore obtain an isomorphism between the cokernels, $\Ext^2_\Pi(X,Y)\cong D\Hom_\Pi(Y,X)$, which is functorial in $X$. To see that the isomorphism is independent of the chosen approximation consider first an exact commutative diagram
\[ \begin{tikzcd}
0 \arrow[r] & W \arrow[r] \arrow[d] & Z \arrow[r] \arrow[d] & Y \arrow[r] \arrow[d,equal] & 0\\
0 \arrow[r] & W' \arrow[r] & Z' \arrow[r] & Y \arrow[r] & 0
\end{tikzcd} \]
with $W,W',Z,Z'\in\fd_{\mathcal I}\Pi$. Again, the fact that the we have a natural isomorphism of bifunctors shows that we obtain a commutative cube on the left, and so the two induced isomorphisms $\Ext^2_\Pi(X,Y)\iso D\Hom_\Pi(Y,X)$ agree. In general, given two approximations $Z\twoheadrightarrow Y$ and $Z'\twoheadrightarrow Y$ with $Z,Z'\in\fd_{\mathcal I}\Pi$, we can compare the both to the approximation $Z\oplus Z'\twoheadrightarrow Y$, proving the claim.

Finally, to see that the construction is functorial in $Y$, consider a morphism $Y\to Y'$, take an approximation $Z'\twoheadrightarrow Y'$ with $Z'\in\fd_{\mathcal I}\Pi$ and form the pullback $\bar Z$. This is again finite dimensional, so can take an approximation $Z\twoheadrightarrow\bar Z$ with $Z\in\fd_{\mathcal I}\Pi$. We have thus constructed an exact commutative diagram
\[ \begin{tikzcd}
0 \arrow[r]  W \arrow[r] \arrow[d] & Z \arrow[r] \arrow[d] & Y \arrow[r] \arrow[d,equal] & 0\\
0 \arrow[r] & W' \arrow[r] & Z' \arrow[r] & Y' \arrow[r] & 0
\end{tikzcd} \]
from which we can deduce that we have a commutative square
\[ \begin{tikzcd}
\Ext^2_\Pi(X,Y) \arrow[r,"\wr"] \arrow[d] & D\Hom_\Pi(Y,X) \arrow[d]\\
\Ext^2_\Pi(X,Y') \arrow[r,"\wr"] & D\Hom_\Pi(Y',X)
\end{tikzcd} \]
This completes the proof.
\end{proof}

\section{Nilpotent modules and Koszul duality}

In this section we restrict ourselves to preprojective algebras of locally finite hereditary tensor algebras of non-Dynkin type. In this case the hereditary algebra is $\Lambda=T_A(M)$, and $\Pi=T/\langle c\rangle$ is the quotient of the hereditary tensor algebra $T=T_A(M\oplus DM)$ by the ideal generated by $c=\rho-\ell$. For convenience we write $\widetilde M\coloneqq M\oplus DM$, and all unadorned tensor products will be over the semisimple algebra $A$.

We note that $\Lambda$ and $T$ are locally finite hereditary tensor algebras with respect to the grading such that $\deg A=0$ and $\deg\widetilde M=1$. As $c$ is homogeneous of degree 2 the preprojective algebra $\Pi$ inherits this grading. Understanding this graded structure on $\Pi$ will be the main focus of this section. As we will only be interested in this grading in this section, we will reappropriate our earlier notation and write $\Pi_r$ for the graded pieces.

One aspect we will investigate will be the Hilbert series of the preprojective algebra. As $\Pi$ is non-commutative and $A$ is in general only semisimple, we define this to be
\[ H(\Pi,t) = \sum_r\sum_{i,j}(\dim e_i\Pi_re_j)t^r \in M_n(\bZ)\llbracket t\rrbracket. \]
As we shall see, the matrices $(\dim e_i\Pi_re_j)$ are related to the polynomials $V_r(x)\in\bZ[x]$ defined recursively via
\[ V_0(x) = 1, \quad V_1(x) = x, \quad V_{r+1}(x) = xV_r(x)-V_{r-1}(x). \]
Their generating function is thus
\[ \sum_r V_r(x)t^m = (1-xt+t^2)^{-1}. \]
The $V_r$ are sometimes called the Vieta--Fibonacci polynomials, or Dickson polynomials of the second kind (with parameter 1), and are related to the Chebyshev polynomials $U_r$ of the second kind via $U_r(x)=V_r(2x)$.

We also observe that the simple graded modules are precisely the simples $S_i$, so a finite dimensional module can be graded if and only if it is nilpotent. Restriction of scalars yields the isomorphism $K_0(\nilp\Pi)\iso K_0(A)$, which we equip with the symmetric $(-,-)$ represented by matrix $2\mat D-\mat B$. We set $\bar{\mat B}\coloneqq \mat D^{-1}\mat B$, and note that this is still an integer matrix.

\subsection{The Hilbert series}

In this section we assume that $\Pi$ is a basic indecomposable preprojective algebra of non-Dynkin type.

\begin{thm}\label{thm:Koszul}
The preprojective algebra $\Pi$ is Koszul. Moreover, $(\dim e_i\Pi_re_j)=\mat DV_m(\bar{\mat B})$ for all $m$, and the Hilbert series of $\Pi$ is given by
\[ H(\Pi,t) = \mat D(1-\bar{\mat B}t+t^2)^{-1}. \]
\end{thm}

\begin{proof}
By \intref{Theorem}{thm:gl-dim-2} we have the standard resolution of $A$
\[ \begin{tikzcd}
0 \arrow[r] & \Pi \arrow[r] & \widetilde M\otimes\Pi \arrow[r] & \Pi \arrow[r] & A \arrow[r] & 0
\end{tikzcd} \]
Next, given an $A$-module $W$, we can regard $W\otimes\Pi$ as a graded projective module generated in degree zero, in which case the standard resolution of $A$ becomes a linear resolution
\[ \begin{tikzcd}
0 \arrow[r] & \Pi(-2) \arrow[r,"\varepsilon"] & (\widetilde M\otimes\Pi)(-1) \arrow[r,"\delta"] & \Pi \arrow[r] & A\arrow[r] & 0.
\end{tikzcd} \]
In other words, the projectives are generated in degrees zero, one and two, and all the morphisms have degree zero. It follows that $\Pi$ is Koszul.

For $r\geq1$ the homogeneous part of degree $r+1$ yields an exact sequence of $A$-bimodules
\[ \begin{tikzcd}
0 \arrow[r] & \Pi_{r-1} \arrow[r] & \widetilde M\otimes\Pi_r \arrow[r] & \Pi_{r+1} \arrow[r] & 0.
\end{tikzcd} \]
Multiplying by idempotents and taking dimensions gives
\begin{multline*}
\dim(e_i\Pi_{r+1}e_j) + \dim(e_i\Pi_{r-1}e_j)\\
= \dim(e_i\widetilde M\otimes\Pi_re_j) = \sum_pb_{ip}d_p^{-1}\dim(e_p\Pi_re_j).
\end{multline*}
Clearly $\Pi_0=A$, so $(\dim e_i\Pi_0e_j)=\mat D$, and $\Pi_1=\widetilde M$, so $(\dim e_i\Pi_1e_j)=\mat B$. Setting $L_r\coloneqq\mat D^{-1}(\dim e_i\Pi_re_j)$ we therefore get $L_0=1$, $L_1=\bar{\mat B}$ and
\[ L_{r+1}+L_{r-1} = \mat D^{-1}\mat BL_r = \bar{\mat B}L_r \quad\textrm{for all }r\geq1. \]
We deduce that $L_r=V_r(\bar{\mat B})$, so that $\Pi$ has Hilbert polynomial
\[ H(\Pi,t) = \mat D\sum V_r(\bar{\mat B})t^r = \mat D(1-\bar{\mat B}t+t^2)^{-1}. \qedhere \] 
\end{proof}

\begin{prop}
For each $r$ the socle and radical series of $\Pi/\Pi_+^r$ coincide, and its socle $\Pi_{r-1}$ is sincere.
\end{prop}

\begin{proof}
Consider an element $y\in\Pi$, which we decompose into its homogeneous parts $y=\sum y_t$. As the radical of $\Pi/\Pi_+^r$ is generated by the image of $\Pi_1=\widetilde M$, we see that the image of $y$ lies in the right socle if and only if $y_t\widetilde M=0$ for all $t<r-1$. Next, in the standard projective resolution of $A$, the left-most map $\varepsilon\colon\Pi\to\Pi\otimes\widetilde M$ is given by $\pi\mapsto\pi\rho-\pi\ell$. As this is a monomorphism we see that $y_t\widetilde M=0$ implies $y_t=0$. Thus the right socle equals $\Pi_+^{r-1}/\Pi_+^r$, which is of course the $(r-1)$-st radical. Similarly for the left socle.

Taking the quotient by the socle we obtain the algebra $\Pi/\Pi_+^{r-1}$, so it follows by induction that the socle and radical series of $\Pi/\Pi_+^r$ agree.

Note that $\soc\Pi/\Pi_+^{r+1}\cong\Pi_r$ as $A$-bimodules, so it remains to show that each $\Pi_r$ is sincere. By \intref{Theorem}{thm:Koszul} we know that $(\dim e_i\Pi_re_j)=\mat DV_r(\bar{\mat B})$. Now $\bar{\mat B}^t=(\mat D^{-1}\mat B)^t$, so each matrix $\mat D\bar{\mat B}^t$ is symmetric, and hence so too is each $\mat DV_r(\bar{\mat B})$. Thus if $\Pi_r$ is insincere as a right $A$-module, then some column vanishes, so by symmetry also some row vanishes, say $e_i\Pi_m=0$. Thus $\rad^r(e_i\Pi/\Pi_+^{r+1})=0$, so using that the socle and radical series agree we deduce that $\soc(e_i\Pi/\Pi_+^{r+1})=0$, a contradiction. We conclude that each $\Pi_r$ is sincere.
\end{proof}

\subsection{The Koszul dual}

The Koszul dual of $\Pi$ is the ext-algebra $\Pi^!\coloneqq\Ext^\ast_\Pi(A,A)$, where we regard $A$ as a right $\Pi$-module. Applying $\Hom_\Pi(-,A)$ to the linear presentation of $A$, all morphisms become zero, so $\Pi^!$ has graded pieces
\[ \Pi^!_0 \cong \End_{\,\hyp A}(A) \cong A, \qquad \Pi^!_1 \cong \Hom_{\,\hyp A}(\widetilde M,A) \cong \widetilde M, \qquad \Pi^!_2 \cong \End_{\,\hyp A}(A) \cong A. \]
Note that the isomorphism $\widetilde M\iso\Hom_{\,\hyp A}(\widetilde M,A)$ is given by combining
\[ DM \iso \Hom_{\,\hyp A}(M,A), \quad \mu\mapsto\rmod\mu, \]
and
\[ M \iso \Hom_{\,\hyp A}(DM,A), \quad m\mapsto\big(\nu\mapsto\lmod\nu(m)\big) \]
from \intref{Lemma}{lem:A-homs}. Using the tensor-hom adjunction we also regard this as a map
\[ \Theta \colon \widetilde M\otimes\widetilde M \to A, \quad (m+\mu)\otimes(n+\nu) = \lmod\nu(m)+\rmod\mu(n). \]

\begin{lem}
The isomorphism $\widetilde M\iso\Ext^1_\Pi(A,A)$ sends $\gamma\in\widetilde M$ to the class of the short exact sequence
\[ \begin{tikzcd}
0 \arrow[r] & A \arrow[r,"\binom01"] & E_\gamma \arrow[r,"{(1,0)}"] & A \arrow[r] & 0
\end{tikzcd} \]
where $E_\gamma=A\oplus A$ as a right $A$-module, with $\Pi$-action
\[ (a,b)\cdot\gamma' \coloneqq (0,\Theta(\gamma\otimes a\gamma')) \quad\textrm{for }\gamma'\in\widetilde M. \]
\end{lem}

\begin{proof}
We just need to check that the following diagram commutes
\[ \begin{tikzcd}
& \widetilde M\otimes\Pi \arrow[r,"\delta"] \arrow[d,"\Theta(\gamma\otimes-)"] & \Pi \arrow[r] \arrow[d,"\psi"] & A \arrow[r] \arrow[d,equal] & 0\\
0 \arrow[r] & A \arrow[r] & E_\gamma \arrow[r] & A \arrow[r] & 0
\end{tikzcd} \]
The vertical map in the middle sends $1\in\Pi$ to $(1,0)\in E_\gamma$, whereas that on the left sends $\gamma'\otimes1$ to $\Theta(\gamma\otimes\gamma')$. Also, the map $\delta$ is just the multiplication map (since $A$ is killed by the graded radical $\Pi_+$). Thus the left hand square commutes, and sends $\gamma'\otimes1\in\widetilde M\otimes\Pi$ to $(0,\Theta(\gamma\otimes\gamma'))\in E_\gamma$. The commutativity of the right hand square is clear.
\end{proof}

We can now compute the Yoneda cup product
\[ \Ext^1_\Pi(A,A) \times \Ext^1_\Pi(A,A) \to \Ext^2_\Pi(A,A) \]
explicitly as a map $\widetilde M\times\widetilde M\to A$.

\begin{lem}\label{lem:yoneda-cup}
The algebra structure on the Koszul dual $\Pi^!$ is induced by the $A$-bilinear map
\[ \Psi \colon \widetilde M\otimes\widetilde M\to A, \quad (m+\mu)\otimes(n+\nu) \mapsto \lmod\nu(m)-\rmod\mu(n). \]
\end{lem}

\begin{proof}
Given $\gamma,\gamma'\in\widetilde M$, we need to construct an exact commutative diagram of the following form, where the connecting map $E_\gamma\to E_{\gamma'}$ sends $(a,b)$ to $(0,a)$
\[ \begin{tikzcd}
0 \arrow[r] & \Pi \arrow[r,"\varepsilon"] \arrow[d,"\theta"] & \widetilde M\otimes\Pi \arrow[r,"\delta"] \arrow[d,"\phi"] &
\Pi \arrow[r] \arrow[d,"\psi"] & A \arrow[r] \arrow[d,equal] & 0\\
0 \arrow[r] & A \arrow[r] & E_\gamma \arrow[r] & E_{\gamma'} \arrow[r] & A \arrow[r] & 0
\end{tikzcd} \]
The morphism $\psi$ is as above, so sends $1$ to $(1,0)\in E_{\gamma_1}$. The composite $\psi\delta$ thus sends $\gamma''\otimes1$ to $(0,\Theta(\gamma'\otimes\gamma''))$, which we lift to the map $\phi\colon\widetilde M\otimes\Pi\to E_\gamma$, $\gamma''\otimes1\mapsto(\Theta(\gamma'\otimes\gamma''),0)$.

We now compute $\phi(\rho)$ and $\phi(\ell)$. For convenience take $\gamma'=n+\nu$. Then, using that $\rho=\sum m_i\otimes\mu_i$ corresponds to the identity of ${}_A(DM)$, so $\nu=\sum\rmod\nu(m_i)\mu_i$, we have
\begin{multline*}
 \phi(\rho) = \sum(\Theta(\gamma'\otimes m_i),0)\cdot\mu_i = \sum(\rmod\nu(m_i),0)\cdot\mu_i\\
= \sum(0,\Theta(\gamma\otimes\rmod\nu(m_i)\mu_i)) = (0,\Theta(\gamma\otimes\nu)) \in E_\gamma.
\end{multline*}
Similarly, $\ell=\sum\nu_j\otimes n_j$ corresponds to the identity on ${}_AM$, and $\phi(\ell)=(0,\Theta(\gamma\otimes n))$.

Finally, we have $\varepsilon(1)=\rho-\ell$, so $\theta(1)=\Theta(\gamma\otimes\nu)-\Theta(\gamma\otimes n)=\Psi(\gamma\otimes\gamma_1)$.
\end{proof}

We wish to give two other constructions of the Koszul dual $\Pi^!$. The first of these is as a trivial extension algebra. Recall that if $\Gamma$ is an algebra and $N$ a $\Gamma$-bimodule, then the trivial extension algebra $\Gamma\ltimes N$ has underlying additive group $\Gamma\oplus N$ with multiplication $(a,m)\cdot(b,n)\coloneqq(ab,an+mb)$. Alternatively, this is the quotient of the tensor algebra $T_\Gamma(N)$ by the square of the graded radical.

We set $\bar\Lambda\coloneqq\Lambda/\Lambda_+^2=A\ltimes M$.

\begin{lem}\label{lem:Koszul}
Using the $A$-bimodule identification $D\bar\Lambda=DM\oplus A$, the natural left and right $\bar\Lambda$-module structures are given respectively by
\[ (a,m)\cdot(\mu,b) \coloneqq (a\mu,ab+\lmod\mu(m)) \quad\textrm{and}\quad (\mu,b)\cdot(a,m) \coloneqq (\mu a,ba+\rmod\mu(m)). \]
\end{lem}

\begin{proof}
To compute the left $\bar\Lambda$-action on $D\bar\Lambda=DM\oplus DA$ take $m\in M$ and $\mu\in DM$. Then $m\cdot\mu\in DA$ sends $a$ to
\[ \mu(am) = \sigma(\lmod\mu(am)) = \sigma(a\lmod\mu(m)) = (\lmod\mu(m)\sigma)(a). \]
Similarly, $\mu\cdot m\in DA$ sends $a$ to
\[ \mu(ma) = \sigma(\rmod\mu(ma)) = \sigma(\rmod\mu(m)a) = (\sigma\rmod\mu(m))(a). \]
Now, as usual, $\sigma$ yields an $A$-bimodule isomorphism $A\cong DA$. This allows us to identify $D\bar\Lambda=DM\oplus A$, in which case we obtain the required $\bar\Lambda$-bimodule structure.
\end{proof}

We define $(D\bar\Lambda)_{\mathrm{tw}}$ by twisting the right $\bar\Lambda$-action by a minus sign
\[ (\mu,b)\cdot(a,m) \coloneqq (\mu a,ba-\rmod\mu(m)). \]

The second construction of the Koszul dual $\Pi^!$ is as the quadratic algebra $T_A(\widetilde M)/\langle\Ker\Psi\rangle$, again in terms of the map
\[ \Psi \colon \widetilde M\otimes\widetilde M\to A, \quad (m+\mu)\otimes(n+\nu) \mapsto \lmod\nu(m)-\rmod\mu(n). \]

We first note that $\Ker\Psi$ is naturally identified with the dual space $D\Pi_2$. To do this we will need that $\Theta\colon\widetilde M\otimes\widetilde M\to A$ induces an isomorphism
\[ \widetilde M\otimes\widetilde M \iso \Hom_{\,\hyp A}(\widetilde M\otimes\widetilde M,A), \quad \gamma\otimes\gamma' \mapsto \big( x\otimes y\mapsto \Theta(\gamma\Theta(\gamma'\otimes x)\otimes y)\big). \]
That this is an isomorphism follows from the sequence
\begin{multline*}
\widetilde M\otimes\widetilde M \iso \widetilde M\otimes\Hom_{\,\hyp A}(\widetilde M,A) \iso \Hom_{\,\hyp A}(\widetilde M,\widetilde M)\\
\iso \Hom_{\,\hyp A}(\widetilde M,\Hom_{\,\hyp A}(\widetilde M,A)) \iso \Hom_{\,\hyp A}(\widetilde M\otimes\widetilde M,A).
\end{multline*}
In particular, the latter half of this sequence sends the identity on $\widetilde M_A$ to $\Theta$.

\begin{lem}\label{lem:quad-dual}
In the standard resolution of $A$, the homogeneous part of degree two
\[ \begin{tikzcd}
0 \arrow[r] & A \arrow[r,"\varepsilon"] & \widetilde M\otimes\widetilde M \arrow[r,"\mathrm{pr}"] & \Pi_2 \arrow[r] & 0
\end{tikzcd} \]
is a short exact sequence of $A$-bimodules satisfying $\varepsilon(1)=c$. Applying $\Hom_{\,\hyp A}(-,A)$ gives the short exact sequence
\[ \begin{tikzcd}
0 \arrow[r] & D\Pi_2 \arrow[r,"\iota"] & \widetilde M\otimes\widetilde M \arrow[r,"\Psi"] & A \arrow[r] & 0
\end{tikzcd} \]
\end{lem}

\begin{proof}
For the term on the left we use the usual identification $\Hom_{\,\hyp A}(\Pi_2,A)\cong D\Pi_2$, whereas for the term in the middle we use the earlier isomorphism
\[ \widetilde M\otimes\widetilde M \iso \Hom_{\,\hyp A}(\widetilde M\otimes\widetilde M,A), \quad
\gamma\otimes\gamma' \mapsto \big( x\otimes y \mapsto \Theta(\gamma\Theta(\gamma'\otimes x)\otimes y) \big). \]
Then, as in the proof of \intref{Lemma}{lem:yoneda-cup}, we see that the image of $\gamma\otimes\gamma'$ sends $c$ to $\Psi(\gamma\otimes\gamma')$.
\end{proof}

\begin{thm}
The Koszul dual $\Pi^!$ is isomorphic to both of the following algebras.
\begin{enumerate}
\item The trivial extension algebra $\bar\Lambda \ltimes (D\bar\Lambda)_{\mathrm{tw}}$. Explicitly, this is the algebra $A\oplus\widetilde M\oplus A$, with multiplication
\[ (a,\gamma,b)\cdot(c,\gamma',d) \coloneqq (ac,a\gamma'+\gamma c,ad+bc+\Psi(\gamma\otimes\gamma')). \]
\item The quadratic algebra $T_A(\widetilde M)/\langle D\Pi_2\rangle$.
\end{enumerate}
\end{thm}

\begin{proof}
(1) We know that, as an $A$-bimodule, the Koszul dual $\Pi^!=\Ext^\ast_\Pi(A,A)$ is given by $A\oplus\widetilde M\oplus A$. Thus the only non-trivial part of the multiplication is that coming from the Yoneda cup product $\Ext^1_\Pi(A,A)\times\Ext^1_\Pi(A,A)\to\Ext^2_\Pi(A,A)$, which we computed in \intref{Lemma}{lem:yoneda-cup} to be $\Psi$. By \intref{Lemma}{lem:Koszul}, this is precisely the trivial extension algebra $\bar\Lambda\ltimes(D\bar\Lambda)_{\mathrm{tw}}$.

(2) By the universal property of the tensor algebra there is a natural map $T_A(\widetilde M)\twoheadrightarrow\Pi^!$ which is the identity modulo the square of the graded radical. Next, \intref{Lemma}{lem:quad-dual} tells us that $\Psi\colon\widetilde M\otimes\widetilde M\twoheadrightarrow A$ is onto with kernel $D\Pi_2$. It follows that we have a surjection $T_A(\widetilde M)/\langle D\Pi_2\rangle\twoheadrightarrow\Pi^!$, which is furthermore an isomorphism in degrees zero, one and two.

It is therefore enough to check that the left hand side vanishes in degree three, so $\widetilde M\otimes\Ker\Psi+\Ker\Psi\otimes\widetilde M=\widetilde M^{\otimes3}$, or equivalently that the map $\id\otimes\Psi\colon\widetilde M^{\otimes3}\twoheadrightarrow\widetilde M$ is surjective when restricted to $\Ker\Psi\otimes\widetilde M$.

We do this as in the proof of \intref{Lemma}{lem:quad-dual} by applying $\Hom_{\,\hyp A}(-,A)$ to the short exact sequence
\[ \begin{tikzcd}
0 \arrow[r] & \widetilde M \arrow[r,"\varepsilon"] & \widetilde M\otimes\Pi_2 \arrow[r] & \Pi_3 \arrow[r] & 0
\end{tikzcd} \]
given by the homogeneous part of degree three in the standard resolution of $A$. We note that the map on the left can be written as the composite
\[ \begin{tikzcd}[column sep=35pt]
\widetilde M \arrow[r,"\sim"] & A\otimes\widetilde M \arrow[r,rightarrowtail,"\varepsilon\otimes\id"] &
\widetilde M^{\otimes3} \arrow[r,twoheadrightarrow,"\id\otimes\mathrm{pr}"] & \widetilde M\otimes\Pi_2
\end{tikzcd} \]
By definition, $\Hom_{\,\hyp A}(\mathrm{pr},A)$, together with our usual identifications, is just the inclusion $\iota\colon D\Pi_2\rightarrow\widetilde M\otimes\widetilde M$. Thus $\Hom_{\,\hyp A}(\id\otimes\mathrm{pr},A)$ equals $\iota\otimes\id$. Similarly $\Hom_{\,\hyp A}(\varepsilon,A)$ is just $\Psi$, so $\Hom(\varepsilon\otimes\id,A_A)$ becomes $\id\otimes\Psi$. As the composite $(\id\otimes\mathrm{pr})(\varepsilon\otimes\mathrm{id})$ is injective, we conclude that $(\id\otimes\Psi)(\iota\otimes\id)$ is surjective, as required.
\end{proof}

\section{Preprojective algebras of Dynkin type}

We finish by looking at preprojective algebras of Dynkin type. We thus have a (basic) finite dimensional hereditary algebra $\Lambda$ of Dynkin type, so representation finite and necessarily a tensor algebra $\Lambda=T_A(M)$, and we form its preprojective algebra $\Pi$.

As usual we have the associated symmetric bilinear form $(-,-)$ on $K_0(\Pi)\cong\bZ^n$, represented by the matrix $2\mat D-\mat B$. We can write this as $\mat D\mat C$, where $\mat C=2-\bar{\mat B}$ is the symmetrisable Cartan matrix of the same type as $\Lambda$ or $\Pi$. We then have the corresponding Weyl group $W\leq\Aut(\bZ^n)$ generated by the simple reflections $s_i(x)\coloneqq x-\frac1{d_i}(e_i,x)e_i$, and the root system $\Phi=\bigcup_iWe_i$. Note that both $W$ and $\Phi$ are finite sets. Finally, a Coxeter element $c\in W$ consists of the product of all simple reflections in some order. All such Coxeter elements are conjugate, and their common order $h$ is called the Coxeter number.

We begin with the well-known result that preprojective agebras of Dynkin type are self-injective. For completeness we include a proof here.

\begin{thm}\label{thm:self-inj}
Every preprojective algebra of Dynkin type is self-injective. In fact, if $\Pi$ is the preprojective algebra of the hereditary algebra $\Lambda$,  then $e_i\Pi\cong D(\Pi e_j)$ as right $\Pi$-modules if and only if the indecomposable $\Lambda$-modules $e_i\Lambda$ and $D(\Lambda e_j)$ lie in the same $\tau$-orbit.
\end{thm}

\begin{proof}
We know that $\Pi=T_\Lambda(\Pi_1)$ is a tensor algebra, so $\Pi$-modules can be regarded as pairs $(X,f)$ consisting of a $\Lambda$-module $X$ together with a $\Lambda$-homomorphism $f\colon\tau^-X\to X$.

Consider now the indecomposable injective module $D(\Pi e_j)$, and the corresponding indecomposable injective $\Lambda$-module $I\coloneqq D(\Lambda e_j)$. As a $\Lambda$-module $D(\Pi e_j)$ decomposes into the direct sum of $D(\Pi_me_j)\cong D(\tau^{-m}\Lambda e_j)$, which by \intref{Corollary}{cor:dual-AR} is isomorphic to $\tau^mI$. Take $m$ maximal such that $\tau^mI$ is non-zero, so that $\tau^mI\cong e_i\Lambda$ is an indecomposable projective.

For $\theta\in D(\Pi_me_j)$ and $\pi\in\Pi_1$ we have that $\theta\pi\in D(\Pi_{m-1}e_j)$ is the map $x\mapsto\theta(\pi x)$. We deduce that if $\theta\neq0$ and $m\geq1$, then since $\Pi_m=\Pi_1\otimes_\Lambda\Pi_{m-1}$ we must have $\theta\pi\neq0$ for some $\pi\in\Pi_1$. Next, $\tau^rI\otimes_\Lambda\Pi_1\cong\tau^-(\tau^rI)\cong\tau^{r-1}I$, and $\End(\tau^{r-1}I)\cong\End(I)$ is a division ring. Thus the $\Pi$-action on $D(\Pi e_j)$ corresponds to a tuple of non-zero morphisms $\tau^rI\otimes\Pi_1\to\tau^{r-1}I$ for $1\leq r\leq m$, which must therefore be isomorphisms. Thus $D(\Pi e_j)\cong \tau^m I\otimes_\Lambda\Pi\cong e_i\Pi$ is projective.
\end{proof}

In general, for a self-injective algebra $\Gamma$, the map $\varrho$ satisfying $e_i\Gamma\cong D(\Gamma e_{\varrho(i)})$ is called the Nakayama permutation. The next well-known lemma exhibits the relationship between $\varrho$ and the inverse Nakayama functor $\nu^-=\Hom_\Gamma(D\Gamma,-)$.

\begin{lem}\label{lem:Nakayama}
Let $\Gamma$ be a self-injective algebra. Then $\nu^-(e_i\Gamma)\cong e_{\varrho(i)}\Gamma$ and  $\nu^-(S_i)\cong S_{\varrho(i)}$. Moreover, $\varrho$ extends to an automorphism of the ext-quiver of $\Gamma$.
\end{lem}

\begin{proof}
Using $e_i\Gamma\cong D(\Gamma e_{\varrho(i)})$ we compute
\[ \nu^-(e_i\Gamma) \cong \Hom_\Gamma(D\Gamma,D(\Gamma e_{\varrho(i)})) \cong D(D\Gamma\cdot e_{\varrho(i)})
\cong D^2(e_{\varrho(i)}\Gamma) \cong e_{\varrho(i)}\Gamma. \]
Also, as $\Gamma$ is self-injective, $\nu^-$ is an autoequivalence. Thus the simple tops of these two modules are isomorphic, so $\nu^-(S_i)\cong S_{\varrho(i)}$. Similarly, $\Ext^1_\Gamma(S_i,S_j)\cong\Ext^1_\Gamma(S_{\varrho(i)},S_{\varrho(j)})$, so $\varrho$ extends to an automorphism of the ext-quiver.
\end{proof}

\begin{thm}\label{thm:ker-epsilon}
Let $\Pi$ be a preprojective algebra of Dynkin type. Then the third syzygy is naturally isomorphic to the inverse Nakayama functor.
\end{thm}

\begin{proof}
We know that the hereditary algebra $\Lambda=T_A(M)$ is a tensor algebra, and that $\Pi$ is finite dimensional. Thus \intref{Proposition}{prop:ker-on-left-2} applies to show that for each right $\Pi$-module $X$, the left hand map in the standard projective presentation of $X$, namely
\[ X\otimes_A\Pi \to X\otimes_A(M\oplus DM)\otimes_A\Pi, \]
has kernel $\Hom_\Pi(D\Pi,X)=\nu^-X$.
\end{proof}

We next show that all projective $\Pi$-modules have the same Loewy length. In fact, this applies more generally to all graded self-injective algebras.

\begin{prop}
Let $\Gamma$ be a self-injective algebra. Assume further that $\Gamma=\Gamma_0\oplus\Gamma_1\oplus\cdots\oplus\Gamma_{\ell-1}$ is graded such that $\Gamma_0$ is semisimple and $\Gamma_1$ generates the radical. Then every projective module $P$ has Loewy length $\ell$, and $\soc^rP=\rad^{\ell-r}P$, so the socle and radical series coincide.
\end{prop}

\begin{proof}
Consider an indecomposable projective module $e_i\Gamma$, and take $m$ maximal such that $e_i\Gamma_{m-1}\neq0$. As $\Gamma_1$ generates the radical, $e_i\Gamma_{m-1}$ is contained in the socle. On the other hand, $\Gamma$ is self-injective, so $e_i\Gamma\cong D(\Gamma e_{\varrho(i)})$ has simple socle $S_{\varrho(i)}$. Thus $\soc(e_i\Gamma)=e_i\Gamma_{m-1}=\rad^{m-1}(e_i\Gamma)$.

Suppose by induction that $\soc^r(e_i\Gamma)=\rad^{m-r}(e_i\Gamma)$. Take $y\in e_i\Gamma$ and decompose it into its homogeneous pieces $y=\sum y_t$. Then $y\in\soc^{r+1}(e_i\Gamma)$ if and only if $y\Gamma_1\in\soc^r(e_i\Gamma)=\rad^{m-r}(e_i\Gamma)$, which is if and only if $y_t\Pi_1=0$ for all $t<m-r-1$. In other words, $y_t$ lies in the socle $e_i\Gamma_{m-1}$ for all such $t$, which happens if and only if $y_t=0$ for all $t<m-r-1$. We deduce that $\soc^{r+1}(e_i\Gamma)=\rad^{m-r-1}(e_i\Gamma)$.

We next show that there is a path from $i$ to $j$ in the ext-quiver of $\Gamma$ if and only if there exists a path from $j$ to $i$. By induction on the length of the path from $i$ to $j$ we reduce to the case when there is an arrow $i\to j$, equivalently $e_i\Gamma_1e_j\neq0$. Using that $e_i\Gamma$ has simple socle $S_{\varrho(i)}=\rad^{m-1}(e_i\Gamma)=e_i\Gamma_{m-1}$, we see that $e_i\Gamma_1e_j\Gamma_{m-2}e_{\varrho(i)}\neq0$, so there is a path (of length $m-2$) from $j$ to $\varrho(i)$. The same argument also shows that there are paths from $\varrho(i)$ to $\varrho^2(i)$, then to $\varrho^3(i)$, and so on. As $\varrho$ is an automorphism we conclude that there is a path going from $j$ back to $i$.

We continue with the assumption that there is an arrow $i\to j$. Then there is also an arrow $\varrho(i)\to\varrho(j)$, so a non-split extension $0\to S_{\varrho(j)}\to X\to S_{\varrho(i)}\to 0$. Using the injective envelope $D(\Gamma e_{\varrho(j)})$ we can express this as a pullback along some non-zero, and hence injective, morphism $S_{\varrho(i)}\rightarrowtail D(\Gamma e_{\varrho(j)})/S_{\varrho(j)}$. Using the injective envelope of $S_{\varrho(i)}$ we obtain a morphism $D(\Gamma e_{\varrho(j)})/S_{\varrho(j)}\to D(\Gamma e_{\varrho(i)})$ which is the identity on $S_{\varrho(i)}$. In other words we have a morphism $e_j\Gamma\to e_i\Gamma$ which kills $\soc(e_j\Gamma)$ but not $\soc^2(e_j\Gamma)$. The image thus has Loewy length $\LL(e_j\Gamma)-1$, and is a proper submodule of $e_i\Gamma$, showing that $\LL(e_j\Gamma)\leq\LL(e_i\Gamma)$. This holds for all arrows, and since every arrow lies on a cycle, we conclude that each indecomposable projective has the same Loewy length.

Finally, as $\Gamma_{\ell-1}\neq0$, we conclude that each projective has Loewy length $\ell$, and $\Gamma_{\ell-1}=\soc(\Gamma)$.
\end{proof}

We finish with the following theorem, showing how various properties of the Weyl group arise in the structure of the preprojective algebra. In particular, our proof of (4) can be seen as a categorification of the proof given in say Bourbaki \cite[VI \S1.11 Proposition 33]{Bourbaki}. As in the previous section, the statements in the theorem refer to the grading of $\Pi=T_A(M\oplus DM)/\langle c\rangle$ arising from the usual grading on $T_A(M\oplus DM)$, so where $\deg A=0$, $\deg M=1$ and $\deg DM=1$. In the proof, however, we will also need the grading coming from $\Pi=T_\Lambda(\tau^-\Lambda)$. For clarity we will denote the graded pieces by $\Pi_{(m)}$ for the former grading (tensor grading), and by $\Pi_m$ for the latter grading ($\Lambda$-grading), so that $\Pi_{(1)}=M\oplus DM$ whereas $\Pi_1=\tau^-\Lambda$.

\begin{thm}
Let $\Pi$ be an indecomposable preprojective algebra of Dynkin type. Set $n$ to be the rank of $\Pi$, and $h$ the corresponding Coxeter number.
\begin{enumerate}
\item The algebra $\Pi$ has Loewy length $h-1$.
\item There is an exact sequence of graded modules
\[ \begin{tikzcd}[column sep=20pt]
0 \arrow[r] & (\soc\Pi)(-h) \arrow[r] & \Pi(-2) \arrow[r,"\varepsilon"] & (\widetilde M\otimes\Pi)(-1) \arrow[r,"\delta"] & \Pi \arrow[r] & A \arrow[r] & 0.
\end{tikzcd} \]
\item For $m\leq h-1$ we have $(\dim e_i\Pi_{(m)}e_j)=\mat DV_m(\bar{\mat B})$. In particular, the $A$-module $\Pi_{(m)}$ is sincere for $m<h-1$, and $V_{h-1}(\bar{\mat B})=0$. Moreover, the Hilbert polynomial equals
\[ H(\Pi,t) = \mat D(1-\bar{\mat B}t+t^2)^{-1} - (d_i\delta_{i\varrho(i)})t^h. \]
\item If $c$ is any Coxeter element, then it acts on the set of roots $\Phi$ with $n$ orbits, each of size $h$. In particular, $|\Phi|=nh$.
\item The Nakayama permutation $\varrho$ has order one or two.
\end{enumerate}
\end{thm}

In the language of \cite{BBK}, (2) says that $\Pi$ is $(h-2,2)$-Koszul.

\begin{proof}
Recall from \intref{Lemma}{lem:indep-orient} that $\Pi$ is independent of the orientation of $\Lambda$, so we may choose an alternating orientation. We may thus assume that $\Lambda$ is a radical squared zero algebra and every simple is either projective or injective. We now have $M\cdot M=0$, and so also $DM\cdot DM=0$ in $\Pi$. Thus $\Pi$ is spanned by alternating products in $M$ and $DM$, and the Loewy length $\LL(\Pi)$ is one more than the maximal length of such an alternating product. Using the $\Lambda$-grading we can keep track of the number of copies of $DM$, so
\[ \LL(e_i\Pi) = \sum_r\LL_\Lambda(e_i\Pi_r) = \sum_r\LL_\Lambda(\tau^{-r}(e_i\Lambda)). \]
Let us write $m_i$ for the size of the $\tau$-orbit containing $e_i\Lambda$, so that $\tau^{1-m_i}(e_i\Lambda)$ is an indecomposable injective module. Using that each non-simple $\tau^{-r}(e_i\Lambda)$ has Loewy length two we see that $\LL(\Pi)$ is $2m_i$ minus the number of simples in the $\tau$-orbit. By Gabriel's Theorem there is a bijection between the isomorphism classes of indecomposable $\Lambda$-modules and the set of positive roots, and each indecomposble $\Lambda$-module is uniquely of the form $\tau^{-r}(e_i\Lambda)$. Thus $\sum_im_i=|\Phi|/2$ and, as there are $n$ simples and each projective has the same Loewy length, we get
\[ \LL(\Pi) = \frac{1}{n}\sum_i\LL(e_i\Pi) = \frac{|\Phi|}{n}-1. \]

Next, passing to the Grothendieck group $K_0(\Pi)$, the orientation of $\Lambda$ determines both the Euler form $\langle-,-\rangle$ and a Coxeter element $c$ satisfying $\langle x,c(y)\rangle=-\langle x,y\rangle$. It follows that $[\tau^-X]=c^-[X]$ for all indecomposable non-injective modules $X$, whereas $c^-[D(\Lambda e_j)]=-[e_j\Lambda]$ for all $j$. See for example \cite{HR}. Thus $c^{-m_i}[e_i\Lambda]=-[e_{\varrho(i)}\Lambda]$, so setting $\ell_i\coloneqq m_i+m_{\varrho(i)}$ gives $c^{-\ell_i}[e_i\Lambda]=[e_{\varrho^2(i)}]$.

We observe that each $\tau$-orbit contains either zero, one or two simples. Suppose first that the $\tau$-orbit of $e_i\Lambda$ contains just one simple module. Then $\LL(e_i\Pi)=2m_i-1$, so that $m_i=|\Phi|/2n$. Now each $\LL(e_j\Pi)$ must be odd, so each $\tau$-orbit contains precisely one simple, and $m_j=|\Phi|/2n$ for all $j$. Suppose instead that the orbit of $e_i\Lambda$ contains two simples, so $m_i=\frac12(1+|\Phi|/n)$. We must have that both $e_i\Lambda$ and $D(\Lambda e_{\varrho(i)})$ are simple, so $e_{\varrho(i)}\Lambda$ is not simple. It follows that the $\tau$-orbit of $e_{\varrho(i)}\Lambda$ contains no simples, and $m_{\varrho(i)}=m_i-1$. Now each projective module has even Loewy length, so each $\tau$-orbit either contains two simples and has length $m_i$, or else contains no simples and has length $m_i-1$. Also, $\varrho$ must swap between such orbits, showing that $\ell_j=2m_i-1=|\Phi|/n$ is independent of $j$.

Setting $\ell\coloneqq|\Phi|/n$ we always have $c^{-\ell}[e_i\Lambda]=[e_{\varrho^2(i)}\Lambda]$, and the $\tau$-orbits of $e_i\Lambda$ and $e_{\varrho^2(i)}\Lambda$ have the same size. We can once again use Gabriel's Theorem to deduce that $c^{-\ell}$ preserves the set of positive roots, so must be the identity. For, every positive root is uniquely of the form $c^{-r}[e_i\Lambda]=[\tau^{-r}e_i\Lambda]$, in which case applying $c^{-\ell}$ yields the positive root $c^{-r}[e_{\varrho^2(i)}\Lambda]=[\tau^{-r}e_{\varrho^2(i)}\Lambda]$. It follows that $\varrho^2=\id$. 

Finally, by construction $\ell$ is the first power of $c^-$ which sends the class of an indecomposable projective to the class of another indecomposable projective. Thus the order $h$ of $c$ is precisely $\ell=|\Phi|/n$. It now follows that $c$ has exactly $n$ orbits in $\Phi$, each of size $h$. This completes the proof of parts (1), (4) and (5).

(2) We recall from \intref{Theorem}{thm:ker-epsilon} that the kernel of $\varepsilon$ in the standard presentation of $A$ is $\nu^-A=\Hom_\Pi(D\Pi,A)$. We now compute
\[ \Hom_\Pi(D\Pi,A) \cong \Hom(D\Pi/\rad(D\Pi),A) \cong \Hom_\Pi(D(\soc({}_\Pi\Pi)),A) \cong \soc\Pi, \]
using that the left and right socles agree, both being given by $\Pi_{(h-1)}$. Thus the kernel on the left of the standard presentation for $A$ is the inclusion $\soc\Pi\rightarrowtail\Pi$. Keeping track of the gradings we obtain a linear presentation of $A$ whose kernel on the left becomes the inclusion $(\soc\Pi)(-h)\rightarrowtail\Pi(-2)$.

(3) This follows as in the proof of \intref{Theorem}{thm:Koszul}. More precisely, we use the linear presentation of $A$, which is exact in degrees $0\leq r\leq h$, to obtain the same recursion for the matrices $(\dim e_i\Pi_{(r)}e_j)$. It follows that $(\dim e_i\Pi_{(r)}e_j)=\mat DV_n(\bar{\mat B})$ for all $r\leq h-1$. The exactness of the linear presentation in degree $h-1$ now gives $\Pi_{h-3}\iso\widetilde M\otimes\Pi_{h-2}$, and hence $V_{h-1}(\bar{\mat B})=0$.

The Hilbert polynomial is therefore $H(\Pi,t)=\sum_{r=0}^{h-2}\mat DV_r(\bar{\mat B})t^r$, so multiplying by $(1-\bar{\mat B}t+t^2)$ and using that $V_{h-1}(\bar{\mat B})=0$ we obtain the matrix $1-\mat DV_{h-2}(\bar{\mat B})t^h$. It remains to note that $\mat DV_{h-2}(\bar{\mat B})=(\dim e_i(\soc\Pi)e_j)$, which as $\soc(e_i\Pi)=S_{\varrho(i)}$ equals $(d_i\delta_{i\varrho(i)})$.
\end{proof}


\begin{thebibliography}{99}
\bibitem{BGL} Baer, D., Geigle, W., and Lenzing, H., `The preprojective algebra of a tame hereditary Artin algebra', Comm. Algebra 15 (1987), 425--457.

\bibitem{BGS} Beilinson, A., Ginzburg, V., and Soergel, W., `Koszul duality patterns in representation theory', J. Amer. Math. Soc. 9 (1996), 473--527.

\bibitem{Bongartz} Bongartz, K., `A generalization of a theorem of M. Auslander', Bull. London Math. Soc. 21 (1989), 255--256.

\bibitem{Bourbaki} Bourbaki, N., Groupes et alg\`ebres de Lie, Chap. 4--6, Springer Berlin, Heidelberg 2006.

\bibitem{BBK} Brenner, S., Butler, M.C.R, and King, A.D., `Periodic algebras which are almost Koszul', Algebras and Representation Theory 5 (2002), 331--367.

\bibitem{Chase} Chase, S. U., `A generalization of the ring of triangular matrices', Nagoya Math. J. 18 (1961), 13--25.

\bibitem{CB1} Crawley-Boevey, W., `Geometry of the moment map for representations of quivers', Compos. Math. 126 (2001), 257--293.

\bibitem{CB2} Crawley-Boevey, W., `On matrices in prescribed conjugacy classes with no common invariant subspace and sum zero', Duke Math. J. 118 (2) 339--352.

\bibitem{CBEG} Crawley-Boevet, W., Etingof, P., Ginzburg, V., `Noncommutative geometry and quiver algebras', Adv. Math. 209 (2007),274--336.

\bibitem{CBH} Crawley-Boevey, W., and Holland, M.P., `Noncommutative deformations of Kleinian singularities', Duke Math. J. 92 (1998), 605--635.

\bibitem{CBHu} Crawley-Boevey, W., and Hubery, A., `The Deligne--Simpson Problem', preprint.

\bibitem{DR1} Dlab, V., and Ringel, C.M., `On algebras of finite representation type', J. Algebra 33 (1975), 306--394.

\bibitem{DR2} Dlab, V., and Ringel, C.M., `Indecomposable representations of graphs and algebras', Mem. Amer. Math. Soc. 6 (1976), no. 173, v+57 pp.

\bibitem{DR3} Dlab, V., and Ringel, C.M., `The representations of tame hereditary algebras', in: Representation theory of algebras (Proc. Conf., Temple Univ., Philadelphia, Pa., 1976), 329--353, Lect. Notes Pure Appl. Math. 37 (Dekker, New York, 1978).

\bibitem{DR4} Dlab, V., and  and Claus Michael Ringel. `The preprojective algebra of a modulated graph', in: Representation Theory
II (Springer Berlin Heidelberg, 1980), 216--231. 

\bibitem{ES} Erdmann, K., and Schroll, S., `Chebyshev polynomials on symmetric matrices', Linear Algebra Appl. 434 (2011), 2475--2496. 

\bibitem{EE} Etingof, P., and Eu, C.-H., `Koszulity and the Hilbert series of preprojective algebras', Math. Res. Lett. 14 (2007), 589--596.

\bibitem{EE2} Etingof, P., and Eu, C.-H., `Hochschild and cyclic homology of preprojective algebras of ADE quivers', Moscow Math. J. 7 (2007),  601--612.


\bibitem{GLS} Geiss, C., Leclerc, B., Schröer, J., `Preprojective algebras and cluster algebras', in: Trends in representation theory of algebras and related topics (Torun, 2007), 253--283, EMS Series of Congress Reports (EMS, Z\"urich, 2008).

\bibitem{GP} Gelfand, I.M., and Ponomarev, V.A., `Model algebras and representations of graphs', Funct. Anal. Appl. 13 (1979), 157--166.

\bibitem{Ginzburg} Ginzburg, V., `Lectures on noncommutative geometry', \url{arxiv.org/abs/math/0506603v1} .

\bibitem{Grant} Grant, J., `The Nakayama automorphism of a self-injective preprojective algebra', Bull. London Math. Soc. 52 (2020), 137--152. 

\bibitem{HR} Happel, D., and Ringel, C.M., `Tilted algebras', Trans. Amer. Math. Soc. 274 (1982), 399--443.

\bibitem{Hubery} Hubery, A., `Connections on sheaves on weighted projective lines, and an application to the Deligne–Simpson Problem', Habilitationsschift, Universit\"at Bielefeld, 2003.

\bibitem{HK} Hubery, A., and Krause, H., `A categorification of non-crossing partitions', J. Eur. Math. Soc. 18 (2016), 2273--2313.

\bibitem{Keller} Keller, B., `Calabi-Yau triangulated categories', in: Trends in representation theory of algebras and related topics, ed. A. Skowro\'nski (Eur. Math. Soc., Z\"urich, 2008), 467--489.

\bibitem{Leclerc} Leclerc, B., `Cluster algebras and representation theory', in: Proc. Internat. Conf. Math. Volume IV (Hindustan Book, New Delhi, 2010), 2471--2488. 

\bibitem{LdlP} Lenzing, H., and de la Pe\~na, J. A., `A Chebysheff recursion formula for Coxeter polynomials', Linear Algebra Appl. 430 (2009), 947--956.

\bibitem{Lusztig} Lusztig, G., `Semicanonical bases arising from enveloping algebras', Adv. Math. 151 (2000), 129--139.

\bibitem{MV} Mart\'inez-Villa, R., `Applications of Koszul algebras: the preprojective algebra', in: Representation theory of algebras (Cocoyoc, 1994), 487--504, CMS Conf. Proc. 18 (Amer. Math. Soc., Providence, RI, 1996).

\bibitem{Nakajima} Nakajima, H., `Quiver varieties and Kac-Moody algebras', Duke Math. 91 (1998), 515--560.

\bibitem{Nakajima2} Nakajima, H., `Introduction to quiver varieties', \url{arxiv.org/abs/1611.10000v1} .

\bibitem{RVdB} Reiten, I., and Van den Bergh, M., `Two-dimensional tame and maximal orders of finite representation type', Mem. Amer. Math. Soc. 80 (1989), n. 408, viii+72 pp.

\bibitem{Sb} S\"oderberg, C., `Preprojective algebras of $d$-representation finite species with relations', J. Pure Applied Algebra 228 (2024). 

\bibitem{Wedderburn} Wedderburn, J. H. M., `On hypercomplex numbers', Proc. London Math. Soc. (2) 6 (1908), 77--118.

\bibitem{Zaks1} Zaks, A., `A note on semi-primary hereditary rings', Pacific J. Math. 23 (1967), 627--628.

\bibitem{Zaks2} Zaks, A., `Semi-primary hereditary rings', Israel J. Math. 6 (1968), 359--362.
\end{thebibliography}
\end{document}